\documentclass[a4paper,10pt]{amsart}
\usepackage{amssymb,amsmath, pinlabel,tikz,accents}

 \usepackage[normalem]{ulem}
  \usepackage{hyperref}
 
\input xy
\xyoption{all}

\usepackage{color}

\newtheorem{theorem}{Theorem}
\newtheorem{lemma}[theorem]{Lemma}
\newtheorem{proposition}[theorem]{Proposition}
\newtheorem{corollary}[theorem]{Corollary}

\theoremstyle{definition}
\newtheorem{definition}[theorem]{Definition}
\newtheorem{notation}[theorem]{Notation}

\theoremstyle{remark}
\newtheorem{remark}[theorem]{Remark}
\newtheorem{example}[theorem]{Example}

\newcommand{\de}{\mathbf{i}}
\newcommand{\ex}{\mathbf{e}}
\newcommand{\ic}{\mathbf{c}}

\newcommand{\gras}[1]{{\mathbb #1}} 
\newcommand{\B}{\gras{B}}
\newcommand{\N}{\gras{N}}
\newcommand{\Z}{\gras{Z}} 
\newcommand{\Q}{\gras{Q}} 
\newcommand{\R}{\gras{R}} 
\newcommand{\C}{\gras{C}}  
\newcommand{\bK}{\gras{K}}

\newcommand{\cA}{\mathcal{A}}
\newcommand{\cB}{\mathcal{B}}
\newcommand{\cD}{\mathcal{D}}
\newcommand{\cE}{\mathcal{E}}
\newcommand{\cF}{\mathcal{F}}
\newcommand{\cH}{\mathcal{H}}
\newcommand{\cI}{\mathcal{I}}
\newcommand{\cL}{\mathcal{L}}
\newcommand{\cN}{\mathcal{N}}
\newcommand{\cO}{\mathcal{O}}
\newcommand{\cV}{\mathcal{V}}

\title[Ultrametric spaces of branches on arborescent singularities]
  {Ultrametric spaces of branches on arborescent singularities}
\author{Evelia R. Garc\'{\i}a Barroso}
\address{Departamento de Matem\'aticas, Estad\'{\i}stica e I.O.
Secci\'on de Matem\'aticas, Universidad de La Laguna. Calle Astrof\'{\i}sico Francisco S\'anchez, 
La Laguna 38200, Tenerife, Espa\~na.}
   \email{ergarcia@ull.es}
\author{Pedro D. Gonz\'alez P\'erez} 
\address{Instituto de Ciencias Matem\'aticas (CSIC-UAM-UC3M-UCM), 
Departamento de \'Algebra , Facultad de Ciencias Matem\'aticas,
Universidad Complutense de Madrid, Plaza de las Ciencias 3, Madrid 28070, Espa\~na.}
   \email{pgonzalez@mat.ucm.es}
\author{Patrick Popescu-Pampu}
   \address{Universit{\'e} Lille, D\'epartement de Maths., B\^atiment M2\\
     Cit\'e Scientifique, 59655, Villeneuve d'Ascq Cedex, France.}
   \email{patrick.popescu@math.univ-lille1.fr}   
     

\begin{document}

{\bf This article appeared in ``\emph{Singularities, Algebraic Geometry, Commutative Algebra  
and Related Topics.  Festschrift for Antonio Campillo on the Occasion of his 65th Birthday.}'' 
G.-M. Greuel, L. Narv\'aez and S. Xamb\'o-Descamps eds., pp. 55--106, Springer, 2018. 
https://doi.org/10.1007/978-3-319-96827-8\_3}
\bigskip

 \maketitle
 
 \bigskip
{\em \hfill This paper is dedicated to Antonio Campillo and Arkadiusz P\l oski. \smallskip} 

\begin{abstract}
      Let $S$ be a normal complex analytic surface singularity. 
    We say that $S$ is \emph{arborescent} if the dual graph of any \emph{good}  resolution of it is a tree. 
    Whenever $A,B$ are distinct branches on $S$, we denote by 
    $A \cdot B$ their intersection number in the sense of Mumford.  
    If $L$ is a fixed branch, we define  
    $U_L(A,B)= (L \cdot A)(L \cdot B)(A \cdot B)^{-1}$ when $A \neq B$ and 
    $U_L(A,A) =0$ otherwise. We generalize  
    a theorem of P\l oski concerning  smooth germs of surfaces,  
    by proving that whenever  $S$ is arborescent, then $U_L$ is an ultrametric 
    on the set of branches of $S$ different from $L$. 
    We compute the maximum of $U_L$, which gives an analog of a theorem of Teissier. 
    We show that $U_L$ encodes topological information about 
    the structure of the embedded resolutions of  any finite set of branches.   This 
    generalizes a theorem of Favre and Jonsson concerning the case when both 
    $S$ and $L$ are smooth.  We generalize also from smooth germs to 
    arbitrary arborescent ones their valuative interpretation of the 
    dual trees of the resolutions of $S$. Our proofs are based in an essential way 
    on a determinantal identity of Eisenbud and Neumann. 
\end{abstract}

\tableofcontents

\pagebreak

\section{Introduction}  
\label{sec:intro}

P\l oski proved the following theorem in his paper \cite{P 85}: 

\medskip
\begin{quotation}
  ``{\em Let $f,g,h \in \C\{x,y\}$ be irreducible power series. Then in the sequence} 
     $$ \frac{m_0(f,g)}{(\mbox{ord } f) \: (\mbox{ord } g)},   \  \  
            \frac{m_0(f,h)}{(\mbox{ord } f) \: (\mbox{ord } h)},  \  \  
              \frac{m_0(g,h)}{(\mbox{ord } g) \: (\mbox{ord } h)} $$ 
    {\em there are two terms equal and the third is not less than the equal terms. 
    Here $m_0(f,g)$ denotes the intersection multiplicity of the branches $f=0, g=0$ and} 
    $\mbox{ord } f$ {\em stands for the order of $f$. }''
\end{quotation}
\medskip

Denote by $Z_f \subset (\C^2, 0)$ the \emph{branch} (that is, the germ of irreducible curve) 
defined by the equation $f =0$. One has analogously the branches $Z_g, Z_h$. 
Looking at the inverses of the previous quotients:
  $$U(Z_f, Z_g) := \left\{ 
        \begin{array}{cl}
               \dfrac{(\mbox{ord } f) \: (\mbox{ord } g)}{m_0(f,g)}, &  \mbox{if } Z_f \neq Z_g\\
               & \\
               0,   &   \mbox{if } Z_f = Z_g 
         \end{array}   \right. ,$$
one may express P\l oski's  theorem in the following equivalent way:
   
\begin{quotation}
  {\em $U$ is an ultrametric distance on the set of branches of $(\C^2,0)$. }
\end{quotation}

Note that, if $Z_f$ and $Z_g$  are two different branches and if $l \in \C\{x,y\}$ defines 
a smooth branch transversal to $Z_f$ and $Z_g$ 
(that is, if one has $\mbox{ord} \: f = m_0(l, f)$ and $\mbox{ord} \: g = m_0(l, g)$), 
then:
  $$U(Z_f, Z_g) =   \frac{m_0(l, f) \: m_0(l, g)}{m_0(f,g)}.  $$
  
  This is the view-point we take in our paper. \emph{Instead of working with multiplicities, 
  we work with intersection multiplicities (also called intersection numbers) with a fixed branch}. 
  More precisely, we study the properties of the quotients: 
   $$ U_L(A, B):= \left\{ \begin{array}{l}
              \dfrac{(L \cdot A) \: (L \cdot B)}{A \cdot B}, \:  \mbox{ if } A \neq B, \\
              \\
              0 , \:  \mbox{ if } A = B ,
                    \end{array} \right. $$
 when $A$ and $B$ vary in the set  of branches of a normal surface singularity $S$ which 
 are different from a fixed branch $L$.            
In the previous formula, $A \cdot  B \in \Q_+^*$ denotes \emph{Mumford's intersection number} 
of \cite[II.b]{M 61} (see Definition \ref{def:mumfint}). 

   \medskip

We focus on the germs of normal surfaces which have in common with $(\C^2,0)$ the 
following crucial property:  \emph{the dual graphs of their resolutions with simple 
normal crossings are trees}. We call \emph{arborescent} the normal surface singularities 
with this property. Note that in this definition we impose no conditions on the \emph{genera} 
of the irreducible components of exceptional divisors (see also Remark \ref{rem:exarb}).   

We prove that  (see Theorem \ref{thm:ultram}):

\begin{quotation}
  {\em Let $S$ be an arborescent  singularity and let $L$ be a 
  fixed branch on it. Then, $U_L$ is an ultrametric distance on the set of 
  branches of $S$ different from $L$. }
\end{quotation}

Given a finite set $\cF$ of branches on $S$, one gets in this way an infinite family of ultrametric 
distances on it, parametrized by the branches $L$ which do not belong to  $\cF$. 
But, whenever one has an ultrametric on a finite set $\cF$, there is a canonically 
associated rooted tree whose set of leaves is $\cF$ (the interior-rooted tree of 
Definition \ref{def:asstree}). 
In our context, we show that the previous 
infinite family of ultrametrics define all the same \emph{unrooted} tree. This tree may be  
interpreted in the following way using the resolutions of $S$ 
(see Theorem \ref{thm:topint} and Remark \ref{rem:ultradual}): 

\begin{quotation}
  {\em Let $S$ be an arborescent singularity and let 
   $\cF$ be a finite set of branches on it. Let $L$ be another branch, which 
   does not belong to $\cF$. Then the interior-rooted tree associated to $U_L$ 
   is homeomorphic to the union of the geodesics joining the representative 
   points of the branches of $\cF $ in the dual graph of any 
   embedded resolution of the reduced Weil divisor whose branches are the elements 
   of $\cF $.  }
\end{quotation}

Both theorems are based on the following result (reformulated differently in 
Proposition \ref{lem:equalfund}):

\begin{quotation}
  {\em Let $S$ be an arborescent singularity. Consider a 
  resolution  of it with simple normal crossing divisor $E$. 
  Denote by $(E_v)_{v \in \cV}$ the irreducible components of $E$ and by 
  $(\check{E}_v)_{v \in \cV}$ their \emph{duals}, that is, the divisors supported by $E$ 
  such that $\check{E}_v \cdot E_w = \delta_{vw}$, 
  for any $v, w \in \cV$ and $ \delta_{vw}$ denotes the Kronecker delta. 
  Then, for any $u,v,w \in \cV$ such that $v$ belongs 
  to the geodesic of the dual tree of $E$ which joins $u$ and $w$, one has:
     $$(- \check{E}_u \cdot \check{E}_v) (- \check{E}_v \cdot \check{E}_w) = 
          (- \check{E}_v \cdot \check{E}_v) (- \check{E}_u \cdot \check{E}_w)  .$$
     }
\end{quotation}

In turn, this last result is obtained from an identity between determinants of weighted trees 
proved by Eisenbud and Neumann \cite[Lemma 20.2]{EN 85} 
(see Proposition \ref{prop:equlfund} below). 
Therefore, our proof is completely different in spirit from P\l oski's original proof, which used 
computations with Newton-Puiseux series. Instead, we work exclusively with the numbers 
$-\check{E}_u \cdot \check{E}_v$, which are positive and birationally invariant, 
in the sense that they 
are unchanged if one replaces $E_u$ and $E_v$ by their strict transforms on a higher 
resolution (see Corollary \ref{cor:invint}).

\medskip
We were also inspired by an inequality of Teissier \cite[Page 40]{T 77}:

\begin{quotation}
  {\em   Let $S$ be a normal surface singularity with marked point $O$.  If $A, B$ are two 
  distinct branches on it and $m_O$ denotes the multiplicity function at $O$, 
  then one has the inequality:
      $$\frac{m_O(A) \cdot m_O(B)}{A \cdot B} \leq m_O(S).$$
}
\end{quotation}

We prove the following analog of it in the setting of arborescent singularities 
(see Corollary \ref{cor:upbound}, in which we describe also the case of equality):

\begin{quotation}
  {\em  Whenever $L, A, B$ are three pairwise distinct branches on the arborescent  
    singularity $S$ and $E_l$ is the unique component of the exceptional 
    divisor of an embedded resolution of $L$ which intersects the strict transform of $L$, 
    one has: 
        $$U_L(A,B) \leq - \check{E}_l \cdot \check{E}_l.$$}
\end{quotation}

Then, in Theorem \ref{thm:ultraint}, we prove the following analog of the fact that 
$U_L$ is an ultrametric: 

\begin{quotation}
{\em Under the hypothesis that 
\emph{the generic hyperplane section of an arborescent singularity 
$S$ is irreducible}, the function $U_O$ defined by the left-hand 
side of Teissier's inequality is also an ultrametric. }
\end{quotation}

Our approach allows us also to extend from smooth germs  
to arbitrary arborescent singularities $S$ a valuative intepretation 
given by Favre and Jonsson \cite[Theorem 6.50]{FJ 04} 
of the natural partial order on the rooted tree defined by $U_L$.  

Namely, consider a branch $A$ on $S$ different from $L$, and an embedded resolution 
of $L + A$. As before, we denote by $(E_v)_{v \in \cV}$ the irreducible components of 
its exceptional divisor. Look at the dual graph of the total transform of $L+A$ 
as a tree rooted at the strict transform of $L$, and denote by $\preceq_L$ the 
associated partial order on its set of vertices, identified with $\cV \cup \{L, A\}$. 
To each $v \in \cV$ is associated a valuation 
$\mathrm{ord}_v^L$ of the local ring $\mathcal{O}$ of 
$S$, proportional to the divisorial valuation of $E_v$ and normalized relative to $L$. 
 Similarly, to the branch $A$  
is associated a semivaluation $\mathrm{int}_A^L$ of $\mathcal{O}$, which is also 
normalized relative to $L$, defined by $\mathrm{int}_A^L(h) = (A \cdot Z_h) \cdot (A \cdot L)^{-1}$ 
(see Definition \ref{def:exval}). Given two semivaluations 
$\nu_1$ and $\nu_2$, say that $\nu_1 \leq_{val} \nu_2$ if $\nu_1(h) \leq \nu_2(h)$ for all 
$h \in \mathcal{O}$. This is obviously a partial order on the set of semivaluations 
of $\mathcal{O}$. We prove (see Theorem \ref{thm:ordsemival}):

\begin{quotation}
  {\em   For an arborescent singularity, 
     the inequality $\mathrm{ord}_{u}^L \leq_{val} \mathrm{ord}_{v}^L$ is equivalent 
            to the inequality $u \preceq_L v$. Similarly, the inequality   
            $\mathrm{ord}_{u}^L \leq_{val} \mathrm{int}_A^L$  is equivalent 
            to the inequality $ u \preceq_L A$. 
}
\end{quotation}

Even when $S$ is smooth, this result is stronger than the result of Favre and Jonsson, 
which concerns only the case where the branch $L$ \emph{is also smooth}. In  this 
last case, our Theorem \ref{thm:topint} specializes to Lemma 3.69 of \cite{FJ 04}, 
in which $U_L(A,B)$ is  expressed in terms 
of what they call the \emph{relative skewness function} $\alpha_x$ on a 
tree of conveniently normalized valuations (we give more explanations in 
Remark \ref{rem:skewness}).

As was the case for P\l oski's treatment in \cite{P 85}, Favre and Jonsson's study 
in \cite{FJ 04} is based in an important way on Newton-Puiseux series. We avoid  
completely the use of such series and we extend  their results to 
all arborescent singularities, by using instead the dual divisors $\check{E}_u$ defined 
above. Our treatment in terms of the divisors $\check{E}_u$ and the numbers 
$-\check{E}_u \cdot \check{E}_v$ 
was inspired by the alternative presentation of the theory of \cite{FJ 04} given by 
Jonsson in \cite[Section 7.3]{J 15}, again in the smooth surface case. 

As in this last paper, 
our study could be continued by looking at the projective system of embedded 
resolutions of divisors of the form $L+C$, for varying reduced Weil divisors $C$, 
and by gluing accordingly the corresponding ultrametric spaces 
and rooted trees (see Remark \ref{rem:birat}). 
One would get at the limit a description of a quotient 
of the Berkovich space of the arborescent singularity $S$. We decided 
not to do this in this paper, in order to isolate what we believe are the most 
elementary ingredients of such a construction, which do not depend in any 
way on Berkovich theory. 

Let us mention also another difference with the treatments of smooth germs 
$S$ in \cite{FJ 04} and \cite{J 15}. In both references, the authors treat simultaneously 
the relations between triples of functions on their trees (called \emph{skewness}, 
\emph{thinness} and \emph{multiplicity}). Our paper shows that an important part 
of their theory (for instance, the reconstruction of the shape of the trees from valuation 
theory) may be done by looking at only one function (the one they call the skewness).

In the whole paper, we work for simplicity with \emph{complex} normal surface singularities. 
Note that, in fact, our techniques make nearly everything work for singularities 
which are spectra of 
normal $2$-dimensional local rings defined over algebraically closed fields of arbitrary 
characteristic. Indeed, our treatment is based on the fact that the intersection matrix of a 
resolution of the singularity is negative definite (see Theorem \ref{prop:negdef} below), 
a theorem which is true in this greater generality (see Lipman \cite[Lemma 14.1]{L 69}). 
The only exception to this possibility of extending our results to positive 
characteristic is Section \ref{aplloski}, as it uses Newton-Puiseux series, 
which behave differently in positive characteristic (see Remark \ref{Campillo}).

\medskip

The paper is structured as follows.
In Section \ref{sec:genorm} we recall standard facts about Mumford's intersection theory 
on normal surface singularities. In Section \ref{sec:treeultram} we present basic 
relations between \emph{ultrametrics}, \emph{arborescent posets}, 
\emph{hierarchies} on finite sets, \emph{trees}, 
\emph{rooted trees}, \emph{height} and \emph{depth} functions on rooted trees,  
and \emph{additive distances} on unrooted ones. Even if those relations are standard, 
we could not find them formulated in a way adapted to our purposes. For this reason 
we present them carefully. The next two sections contain our results. 
Namely,  the ultrametric spaces of branches 
of arborescent singularities are studied in Section \ref{sec:main}  and the valuative interpretations 
are developed in Section \ref{sec:valcons}.  We dedicate a special subsection 
(\ref{aplloski}) to the original case considered by P\l oski, where both $S$ and $L$ are smooth, 
by giving an alternative proof of his theorem using so-called \emph{Eggers-Wall trees}. 
Finally, Section \ref{sec:oprob} 
contains examples and a list of open problems which 
turn around the following question: \emph{is it possible to extend at least partially 
our results to some singularities which are not arborescent}? Our examples show 
that there exist normal surface singularities and branches $L$ on them 
for which $U_L$ is not even a metric.

\medskip
Since the first version of this paper, we have greatly extended its results,  
in collaboration with Matteo Ruggiero (see \cite{GGPR 17}), 
modifiying also substantially our approach: 
instead of working with rooted trees associated with ultrametrics, we  
work with unrooted trees associated with metrics satisfying the so-called 
{\em four point condition}. Nevertheless, we feel that this first approach remains 
interesting and potentially useful in other contexts.

\section{A reminder on intersection theory for normal surface singularities}
\label{sec:genorm}

In this section we introduce the basic vocabulary and properties needed in the sequel about 
complex normal surface singularities $S$ and about the branches on them. 
In particular, we recall the notions of \emph{good resolution}, 
\emph{associated dual graph}, natural pair of dual lattices and 
\emph{intersection form} on $S$. 
We explain the notion of \emph{determinant} of $S$  
and the way to define, following Mumford, a rational \emph{intersection 
number of effective divisors without common components} on $S$. 
This definition is based in turn on the definition of the \emph{exceptional transform} 
of such a divisor on any resolution of $S$. The exceptional transform belongs to the 
\emph{nef cone} of the resolution. We show that the exceptional divisors belonging to the 
interior of the nef cone are proportional to exceptional transforms of principal divisors 
on $S$, a fact which we use later in the proof of Theorem \ref{thm:ordsemival}.

\subsection{The determinant of a normal surface singularity} 
$\:$ 
\medskip

In the whole paper, $(S,O)$ denotes a {\bf complex analytic normal surface singularity}, 
that is, a germ of complex analytic normal surface.  The germ is allowed to be 
smooth, in which case it will still be called a singularity. This is a common abuse of 
language.  Most of the time we will write simply $S$ instead of $(S,O)$. 

\begin{definition}
A {\bf resolution} of $S$ is a proper bimeromorphic morphism $\pi : \tilde{S} \to S$ 
with total space $\tilde{S}$ smooth. By abuse of language, we will also say in this case 
that $\tilde{S}$ is a resolution of $S$.  
The {\bf exceptional divisor} $E  := \pi^{-1}(O)$ of the resolution is considered 
as a reduced curve.  
A resolution of $S$ is {\bf good} if its exceptional divisor has simple  normal crossings, that is, if all its components are smooth and its singularities 
are ordinary double points. 
\end{definition}

A special case of the so-called \emph{Zariski main theorem} 
(see \cite[Cor. 11.4]{H 77}) implies that $E$ is connected, 
hence the associated \emph{weighted dual 
graph} is also connected:

\begin{definition}  \label{def:dualgraph}
   Let $\pi : \tilde{S} \to S$ be a good resolution of $S$.   We denote the irreducible 
   components of its exceptional divisor $E$ by $(E_u)_{u \in \cV}$.
   The {\bf weighted dual graph} 
   $\Gamma$ of the resolution has $\cV$ as vertex set. There are no loops, but as 
   many edges between the distinct vertices $u,v$ as the intersection number 
   $E_u \cdot E_v$. Moreover, each vertex $v \in \cV$ is weighted by the self-intersection 
   number $E_v \cdot E_v$ of $E_v$ on $\tilde{S}$. 
\end{definition}

Usually one decorates each vertex $u$ also by the genus of the corresponding component 
$E_u$. \emph{We do not do this, because those genera play no role in our study.}

The set $\cV$ may be seen not only as the vertex set of the dual graph $\Gamma$, 
but also as a set of parameters for canonical bases of the two following dual lattices, introduced by 
Lipman in  \cite[Section 18]{L 69}:
\[ \Lambda := \bigoplus_{u \in \cV} \Z  \:  E_u, \:  \:  
        \check{\Lambda}:= \mbox{Hom}_{\Z} ( \Lambda, \Z).\]
   The {\bf intersection form}: 
       $$\begin{array}{cccc} 
              I:    &    \Lambda \times \Lambda &   \to &    \Z   \\ 
                     &    (D_1, D_2) &  \to    &  D_1 \cdot D_2 
            \end{array}$$ 
   is the symmetric bilinear form on $\Lambda$ which computes the intersection 
   number of compact divisors on $\tilde{S}$.   
The following fundamental theorem was proved by Du Val \cite[Section 4]{DV 44} and 
Mumford \cite[Page 6]{M 61}.  
It is also a consequence of Zariski \cite[Lemma 7.4]{Z 62}. See also a proof 
in Lipman \cite[Lemma 14.1]{L 69}, where it is explained that the theorem 
remains true for normal surface singularities defined over arbitrary algebraically 
closed fields, possibly of positive characteristic:

\begin{theorem}  \label{prop:negdef}
  The intersection form of any resolution of $S$ is negative definite. 
\end{theorem}

As a consequence, the map:
   $$\widetilde{I}: \Lambda \to \check{\Lambda}$$
induced by the intersection form $I$ is an embedding of lattices, which 
allows to see $\Lambda$ as a sublattice of finite index of $\check{\Lambda}$, and 
$\check{\Lambda}$ as a lattice of the $\Q$-vector space 
$\Lambda_{\Q} := \Lambda \otimes_{\Z} \Q$. 
In particular, the real vector space $\check{\Lambda}_{\R}$ gets identified with $\Lambda_{\R}$.  
As the intersection form $I$ extends canonically to $\Lambda_{\Q}$, one may restrict it 
to $\check{\Lambda}$. We will denote those extensions by the same symbol $I$.  
The following lemma is immediate:

\begin{lemma}  \label{lem:chardual}
   Seen as a subset of $\Lambda_{\Q}$, the lattice $\check{\Lambda}$ 
   may be characterized in the following way:
     \[
       \check{\Lambda} = \{ D \in \Lambda_{\Q} \:  | \:  D \cdot E_u \in \Z, \mbox{ for all } u \in \cV \}.
      \]
\end{lemma}

Denote by $\delta_{uv}$  the Kronecker delta.  The basis of $\check{\Lambda}$ 
which is dual to the basis 
 $(E_u)_{u \in \cV}$ of $\Lambda$ may be characterized in the following way 
 as a subset of $\Lambda_{\Q}$: 

\begin{definition} \label{def:Lipbasis}
 The {\bf Lipman basis} of $\pi$ is the set $(\check{E}_u)_{u \in \cV}$ of elements of the 
 lattice $\check{\Lambda}$ defined by:
   $$ \check{E}_u \cdot E_v = \delta_{uv}, \mbox{ for all } v \in \cV.$$
\end{definition}

Definition \ref{def:Lipbasis} implies immediately that: 
  \begin{equation} \label{eq:sense1}
        E_u = \sum_{w \in \cV} (E_w \cdot E_u) \check{E}_w, \:  \mbox{ for all } u \in \cV,
  \end{equation}
and that, conversely:
    \begin{equation}  \label{eq:sense2}
          \check{E}_u = \sum_{w \in \cV} (\check{E}_w \cdot \check{E}_u) E_w, \:  
              \mbox{ for all } u \in \cV. 
  \end{equation}
  
 More generally, whenever $D \in \Lambda_{\R}= \check{\Lambda}_{\R}$, one has:
    \begin{equation}  \label{eq:sense3}
         D = \sum_{w \in \cV} (D \cdot \check{E}_w ) E_w = 
              \sum_{w \in \cV} (D \cdot E_w) \check{E}_w. 
  \end{equation}   
  
  The following result is well-known and goes back at least to Zariski \cite[Lemma 7.1]{Z 62} 
  (see it also creeping inside Artin's proof of \cite[Prop. 2 (i)]{A 66}):  
  
  \begin{proposition}  \label{prop:negcoef}
      If $D \in \Lambda_{\Q} \setminus \{0\}$ is such that $D \cdot E_v \geq 0$ 
      for all $v \in \cV$, then all the coefficients $D \cdot \check{E}_v$ of 
      $D$ in the basis $(E_v)_{v \in \cV}$ of $\Lambda_{\Q}$ are negative.  
  \end{proposition}
  
  Combining the equations (\ref{eq:sense1}) and (\ref{eq:sense2}), one gets:
  
  \begin{proposition}  \label{prop:invmat}
     Once one chooses a total order on $\cV$, the matrices $(E_u \cdot E_v)_{(u,v) \in \cV^2}$ 
     and $(\check{E}_u \cdot \check{E}_v)_{(u,v) \in \cV^2}$ are inverse of each other. 
  \end{proposition}
  
  By abuse of language, we say that the two previous functions from $\cV^2$ to $\Q$ 
  are \emph{matrices}, even without a choice of total order on the index set $\cV$.  
  The intersection matrix $(E_u \cdot E_v)_{(u,v) \in \cV^2}$  has negative entries 
  on the diagonal (as a consequence of Theorem \ref{prop:negdef}) and non-negative 
  entries outside it. By contrast: 
  
  \begin{proposition}  \label{prop:stricpos}
     The inverse matrix $(\check{E}_u \cdot \check{E}_v)_{(u,v) \in \cV^2}$ 
      has all its entries negative. 
  \end{proposition}
  
  \begin{proof}
      By formula (\ref{eq:sense2}), the entries of this matrix are the coefficients 
      of the various rational divisors $\check{E}_u$ in the basis $(E_v)_{v \in \cV}$ of 
      $\Lambda_{\Q}$. By Definition \ref{def:Lipbasis}, we see that 
      $\check{E}_u \cdot E_v \geq 0$ for all $v \in \cV$. One may conclude using 
      Proposition \ref{prop:negcoef}. 
  \end{proof}
    
    The fact that if the entries of a symmetric positive definite matrix are non-positive 
    outside the diagonal, then the entries of the inverse matrix are non-negative was proved 
    by Coxeter \cite[Lemma 9.1]{C 34} and differently by Du Val 
    \cite[Page 309]{DV 40}, following a suggestion of Mahler. The stronger fact 
    that in our case the entries of $(- \check{E}_u \cdot \check{E}_v)_{(u,v) \in \cV^2}$  
    are positive comes from the fact that the dual graph of the initial matrix 
    $(- E_u \cdot E_v)_{(u,v) \in \cV^2}$ is connected. 
    For historical details about this theme, see Coxeter \cite[Section 10.9]{C 73}.

  Proposition \ref{prop:negcoef} may be reformulated as the fact that 
  \emph{the pointed nef cone of $\pi$ is included 
  in the interior of the opposite of the effective cone of $\pi$}, 
  where we use the following terminology, which 
  is standard for global algebraic varieties:

  \begin{definition} \label{def:effnef}
     Let $\pi$ be a resolution of $S$. The {\bf effective cone} $\sigma$ of $\pi$ is the simplicial 
     subcone of $\Lambda_{\R}$ consisting of those divisors with non-negative coefficients 
     in the basis $(E_u)_{u \in \cV}$. The {\bf nef cone} $\check{\sigma}$ of $\pi$ is the simplicial 
     subcone of $\check{\Lambda}_{\R}$, identified to $\Lambda_{\R}$ through 
     $\tilde{I}$, consisting of those divisors whose intersections with all effective divisors 
     are non-negative. That is, $\sigma$ is the convex cone generated by 
     $(E_u)_{u \in \cV}$ 
     and $\check{\sigma}$ is the convex cone generated by $(\check{E}_u)_{u \in \cV}$. 
 \end{definition}

The determinant of a symmetric bilinear form on a lattice is well-defined: it is 
   independent of the 
basis of the lattice used to compute it. When the bilinear form is positive definite, 
the determinant is positive. This motivates to look also at the opposite of the intersection 
form (see Neumann and Wahl \cite[Sect. 12]{NW 05}). 
Up to the sign, the following notion was also studied in \cite{EN 85} and \cite{N 89}. 

\begin{definition}  \label{def:detgraph}
    Let $\pi$ be a good resolution of $S$ with weighted dual graph $\Gamma$. 
   The {\bf determinant} $\det(\Gamma) \in \N^*$ of $\Gamma$ is the determinant of  
   the opposite $-I$ of the intersection form, that 
   is, the determinant of the matrix $(-E_u \cdot E_v)_{(u,v) \in \cV^2}$. 
\end{definition}

It is well-known (see for instance \cite[Prop. 3.4 of Chap. 2]{D 92}) that $\det(\Gamma)$ 
is equal to the cardinal  of the torsion subgroup of the 
first integral homology group of the boundary (or link) of $S$. This 
shows that $\det(\Gamma)$ is independent of the choice of 
resolution of $S$. This fact could have been proved directly, by studying the 
effect of a blow-up of one point on the exceptional divisor of a given resolution and 
by using the fact that any two resolutions are related by a finite sequence of 
blow ups of points and of their inverse blow-downs.  As a consequence, we define:

\begin{definition}  \label{def:det(S)ing}
   The {\bf determinant $\det(S)$ of the singularity} $S$ is the determinant of 
   the weighted dual graph of any good resolution of it. 
\end{definition}

\subsection{Mumford's rational  intersection number of branches}
$\:$ 
\medskip

As explained in the introduction, we will be mainly interested by the \emph{branches} 
living on $S$: 

\begin{definition} \label{def:branch}
  A {\bf branch} on $S$ is a germ of irreducible formal curve on $S$. 
  We denote by $\cB(S)$ the set of branches on $S$. 
  A {\bf divisor} (respectively {\bf $\Q$-divisor}) 
  on $S$ is an element of the free abelian group (respectively {\bf $\Q$-vector space}) 
  generated by 
  the branches living on it. The divisor is {\bf effective} if all its coefficients are non-negative. 
  It is {\bf principal} if it is the divisor of a germ of formal function on $S$. 
\end{definition}

Note that the divisors we consider are Weil divisors, as they are not necessarily principal. 

We will study the divisors on $S$ using their \emph{embedded resolutions}, to 
which one extends the notion of \emph{weighted dual graph}: 

\begin{definition}  \label{def:embres}
    If $D$ is a divisor on $S$, an {\bf embedded 
    resolution} of $D$ is a resolution $ \pi : \tilde{S} \to S$ of $S$ such that the preimage  
    $\pi^{-1}| D |$ of $| D|$ on $\tilde{S}$ has simple normal crossings 
    (here $ | D |$ denotes the support of $D$). 
    The {\bf weighted dual graph} $\Gamma_D$ of $D$ with respect to $\pi$ 
    is obtained from the weighted dual 
    graph $\Gamma$ of $\pi$ by adding as new vertices the branches of  $D$ 
    and by joining each such branch $A$ to the unique 
    vertex $u(A) \in \cV$ such that $E_{u(A)}$ meets $A$.   
\end{definition}

The following construction of Mumford \cite{M 61} of a canonical, possibly non-reduced 
structure of $\Q$-divisor on $\pi^{-1}| D |$, will be very important for us:

\begin{definition} \label{def:totransf}
   Let $A$ be a  divisor on $S$ and  $\pi: \tilde{S} \to S$ a resolution 
   of $S$. The {\bf total transform} of $A$ on $\tilde{S}$ is the $\Q$-divisor 
   $\pi^*A = \tilde{A} + (\pi^*A)_{ex}$ on $\tilde{S}$ such that:
     \begin{enumerate}
         \item $\tilde{A}$ is the {\bf strict transform} of $A$ on $\tilde{S}$ (that is, the sum of 
         the closures  inside $\tilde{S}$ of the branches of $A$, keeping unchanged 
         the coefficient of each branch). 
         \item The support of the {\bf exceptional transform} $(\pi^*A)_{ex}$ of $A$ 
              on $\tilde{S}$ (or by $\pi$) 
         is included in the exceptional divisor $E$. 
         \item $\pi^*A\cdot  E_u =0$ for each irreducible component $E_u$ of $E$. 
     \end{enumerate}
\end{definition}

Such a divisor $(\pi^*A)_{ex}$ exists and is unique:

\begin{proposition}  \label{prop:exprexcept}
    Let $A$ be a  divisor on $S$ and  
    $\pi: \tilde{S} \to S$ a resolution  of $S$. Then the exceptional transform 
    $(\pi^*A)_{ex}$ of $A$ is given by the following formula:
       $$(\pi^*A)_{ex} = - \sum_{u \in \cV} (\tilde{A} \cdot E_u) \:   \check{E}_u. $$
    In particular, $(\pi^*A)_{ex}$ lies in the opposite of the nef cone. 
\end{proposition}

\begin{proof}
   The third condition of Definition \ref{def:totransf} implies that:
      $ (\pi^*A)_{ex} \cdot E_u = - \tilde{A} \:  \cdot \: E_u$, for all  $u \in \cV. $
   By combining this with equation (\ref{eq:sense3}), we get:
      \[(\pi^*A)_{ex} =  \sum_{u \in \cV} ((\pi^*A)_{ex} \cdot E_u) \:   \check{E}_u =  
          - \sum_{u \in \cV} (\tilde{A} \cdot E_u) \:   \check{E}_u. \]
   As $\tilde{A} \cdot E_u \geq 0$ for all $u \in \cV$, we see that $- (\pi^*A)_{ex}$ lies 
   in the cone generated by $(\check{E}_u)_{u \in \cV}$ inside $\Lambda_{\R}$ which, 
   by Definition \ref{def:effnef}, is the nef cone $\check{\sigma}$. 
\end{proof}

In the case in which $A$ is principal, defined by a germ of 
holomorphic function $f_A$, then $(\pi^*A)_{ex}$ is 
simply the exceptional part of the principal divisor 
on $\tilde{S}$ defined by the pull-back function $\pi^* f_A$. By Proposition \ref{prop:exprexcept}, 
in this case $(\pi^*A)_{ex}$ belongs to the semigroup $- \check{\sigma} \: \cap \: \Lambda$ of 
integral exceptional divisors whose opposites are nef. In general 
not all the elements of this semigroup 
consist in such exceptional transforms of principal divisors, but this is true for those lying 
in the interior of  $- \check{\sigma}$:

\begin{proposition}  \label{prop:intrealis}
   Consider a resolution of $S$. 
   Any element of the lattice $\Lambda$ which lies in the interior of the cone  $- \check{\sigma}$ 
   has a multiple by a positive integer which is the exceptional transform of an 
   effective principal divisor on $S$. 
\end{proposition}

\begin{proof}  
   Denote by $K$ a canonical divisor on the resolution $\tilde{S}$ and by 
   $E = \sum_{v \in \cV} E_v$ the reduced exceptional divisor.  
   By \cite[Theorem 4.1]{CNP 06}, any divisor $D \in \Lambda$ such that: 
      \begin{equation} \label{eq:fundineq} 
            (D + E + K) \cdot E_u \leq -2   \mbox{ for all } u \in \cV
      \end{equation}
   is the exceptional transform of an effective principal divisor. 
  
   Assume now that $H \in \Lambda$ belongs to the \emph{interior} of the opposite 
   $- \check{\sigma}$ of the nef cone. This means that $H \cdot E_u < 0$ for 
   any $u \in \cV$. There exists therefore $n \in \N^*$ such that:
    $ nH \cdot E_u < - ( E + K) \cdot E_u  -2 $, for all $ u \in \cV$. 
   Equivalently, $D := nH$ satisfies the inequalities (\ref{eq:fundineq}), 
   therefore it is the exceptional transform of an effective principal divisor. 
\end{proof}

Definition \ref{def:totransf} allowed Mumford to introduce a \emph{rational 
intersection number} of any two divisors on $S$ without common branches:

\begin{definition} \label{def:mumfint}
   Let $A, B$ be two divisors on $S$ without common branches. Then their 
   {\bf intersection number} $A \cdot B \in \Q$ is defined by:
       $A \cdot B := \pi^*A \cdot \pi^* B,$
    for any embedded resolution $\pi$ of $A+ B$. 
\end{definition}

The fact that this definition is independent of the resolution was proved by Mumford 
\cite[Section II (b)]{M 61}, 
by showing that it is unchanged if one blows up one point on $E$. 
As an immediate consequence we have:

\begin{proposition}  \label{prop:intexcep}
   Let $A, B$ be two divisors on $S$ without common branches and $(\pi^*A)_{ex}$, 
   $(\pi^*A)_{ex}$ be 
   their exceptional transforms on an embedded resolution $\tilde{S}$ 
   of $A+B$. Then $A \cdot B = - (\pi^*A)_{ex} \cdot  (\pi^*B)_{ex}.$
\end{proposition}

Assume now that $A$ is a branch on $S$. Consider an embedded resolution of it. 
Recall from Definition \ref{def:embres} that $u(A)$ denotes the unique index $u \in  \cV$  
such that the strict transform $\tilde{A}$ meets $E_u$.  Then one has another immediate 
consequence of the definitions:

\begin{lemma}  \label{lem:excdual}
     If  $A$ is a branch on $S$ and if $\tilde{S}$ is an embedded resolution 
   of $A$, then   $(\pi^*A)_{ex} = - \check{E}_{u(A)}$.
\end{lemma}

By combining Proposition \ref{prop:intexcep}, Lemma \ref{lem:excdual} and Proposition 
\ref{prop:stricpos}, one gets:

\begin{corollary}  \label{cor:invint} $\: $   

   \begin{enumerate}
       \item Assume that $A$ and $B$ are distinct branches on $S$. If $\tilde{S}$ 
            is an embedded resolution of $A + B$, then:
              $$A \cdot B = - \check{E}_{u(A)} \cdot \check{E}_{u(B)} >0. $$
         \item If $E_u, E_v$ are two components of the exceptional divisor 
               of a resolution, then the intersection number $\check{E}_u \cdot \check{E}_v$ 
               is independent of the resolution (that is, it will be the same if one 
               replaces $E_u, E_v$ by their strict transforms on another resolution). 
   \end{enumerate}
\end{corollary}

\section{Generalities on ultrametrics and trees }
\label{sec:treeultram}

In this section we explain basic relations between  
\emph{ultrametrics}, \emph{arborescent posets},  
\emph{hierarchies} on finite sets, \emph{rooted} and \emph{unrooted trees}, 
\emph{height} and \emph{depth} functions on rooted trees,  
and \emph{additive distances} on unrooted ones.   The fact that we couldn't find these 
relations described in the literature in a way adapted to our purposes, explains 
the level of detail of this section. The framework developed here allows us  
to formulate in the next two sections simple conceptual proofs of our main results: 
Theorems \ref{thm:ultram}, \ref{thm:topint}, \ref{thm:ultraint}, \ref{thm:ordsemival} 
and Corollary \ref{cor:valtree}.

We begin by explaning the relation between \emph{rooted trees}, 
\emph{arborescent posets} 
and \emph{hierarchies} (see Subsection \ref{trees-arb-hier}). 
In Subsection \ref{ultramdepth} we recall the notion of  \emph{ultrametric spaces} and 
we explain that ultrametrics on finite sets may be alternatively 
presented as rooted trees endowed with a \emph{depth function}, the 
intermediate object in this structural metamorphosis being the \emph{hierarchy}  
of closed balls of the metric, which is an \emph{arborescent poset}. 
We also define a dual notion of \emph{height function} on a tree. 
This allows us to relate in Subsection \ref{addist} ultrametrics with \emph{additive distances} on 
unrooted trees: we explain that any choice of root allows to transform canonically such a 
distance into a height function. Finally, in Subsection \ref{aplloski} 
we apply the previous considerations 
by giving a new proof of P\l oski's original theorem, using the notion of \emph{Eggers-Wall tree}. 
\medskip

\subsection{Trees, rooted trees, arborescent partial orders and hierarchies} 
\label{trees-arb-hier}
$\:$ 
\medskip

If $V$ is a set, we denote by $\binom V2$ the set of its subsets with $2$ elements. 
There are two related notions of trees: combinatorial and topological.

\begin{definition} \label{def:combtree}
   A {\bf combinatorial tree} is a finite combinatorial connected graph without cycles, 
   that is, a pair 
   $(\cV, \cA)$ where 
   $\cV$ is a finite set of {\bf vertices}, $\cA \subset \binom \cV2$ is the set of {\bf edges} 
   and for any two distinct vertices $u,v$, there is a unique sequence 
   $\{v_0, v_1\}, ..., \{v_{k-1} ,v_k\}$ of edges, with $v_0=u, v_k=v$ and 
   $v_0, ..., v_k$ pairwise distinct.    
   Its {\bf geometric realization} is the simplicial 
   complex obtained by joining the vertices  which are end-points of edges 
   by segments in the real vector space $\R^{\cV}$. 
 \end{definition}
  
  \begin{definition} \label{def:toptree}
    A {\bf topological tree} $T$ is a topological space 
   homeomorphic to the geometric realization of a combinatorial tree. 
   If $u,v$ are two points of it, 
   the unique embedded arc joining $u$ and $v$ is called the {\bf geodesic 
   joining $u$ and $v$} and is denoted $[uv]=[vu]$. More generally, the {\bf convex hull} 
   $[V]$ of  a finite subset $V$ of $T$ is the union of the geodesics joining pairwise the 
   points of $V$. If $V=\{u,v,w, ...\}$, we denote also $[V] = [uvw ...]$. 
   The {\bf valency}  of a point 
   $u \in T$ is the number of connected components of its complement 
   $T  \setminus \{u\}$. 
   The {\bf ends} of the topological tree are its points  of valency $1$ and its {\bf nodes} are 
   its points of valency at least $3$.  An underlying combinatorial tree being fixed 
   (which will always be the case in the sequel), its vertices are by definition the 
   {\bf vertices} of $T$. A point $u \in T$ is called {\bf interior} if it is not an end. Denote: 
      \begin{itemize}
          \item $\cA(T) =$ the set of edges of $T$; 
         \item $\cV(T) =$ the set of vertices of $T$; 
         \item $\cE(T)= $ the set of ends of $T$;
         \item $\cN(T)= $ the set of nodes of $T$; 
         \item $\cI(T)= $ the set of interior vertices of $T$.
      \end{itemize}
\end{definition}
 
One has the 
inclusion $\cN(T) \cup \cE(T) \subset \cV(T)$, which is strict if and only if either there is at 
least one  vertex of valency $2$ or the tree is reduced to a point (that is, $\cV(T)$ has only
one element, hence $\cA(T) = \emptyset$). 

\medskip

We will use also the following vocabulary about posets:

\begin{definition} \label{def:pred}
In a poset $(V, \preceq)$, we say that $u\in V$ is a {\bf predecessor} of $v \in V$ 
or that $v$ is a {\bf successor} of $u$ if  $u \prec v$ (which means that the inequality is 
\emph{strict}). 
We say that the predecessor $u$ of $v$ is a {\bf direct predecessor} of $v$ if it  is not a 
predecessor of any other predecessor of $v$. Then we say also that $v$ is a 
{\bf direct successor} 
of $u$. Two elements of $V$ are {\bf comparable} if one is a predecessor of the other, 
and {\bf directly comparable} if one is a direct predecessor of the other one.  
\end{definition}

In the sequel we will work also with \emph{rooted trees}, where the root may be either 
a vertex or a point interior to some edge. A choice of root of a given tree endows it with 
a canonical partial order:

\begin{definition}  \label{def:rootree}
    Let $T$ be a topological tree. One says that it is {\bf rooted} if a point ${\rho} \in T$ 
    is chosen, called the {\bf root}. Whenever we want to emphasize the root, we 
    denote by $T_{\rho}$ the tree $T$ rooted at ${\rho}$.  
    The {\bf associated partial order} $\preceq_{\rho}$  on $T_{\rho}$ is defined by:
        \[u \preceq_{\rho} v \:  \Longleftrightarrow \:   [\rho u] \subset [\rho v].\]
     The maximal elements for this partial order are called the {\bf leaves} of $T_{\rho}$. 
     Their set is denoted by $\cL(T_{\rho})$. 
     The {\bf topological vertex set} $\cV_{top}(T_{\rho})$ of $T_{\rho}$ 
     consists of its leaves, its nodes 
     and its root. That is, it is defined by: 
        \begin{equation} \label{eq:topvertex}
        \cV_{top}(T_{\rho}) := \cL(T_{\rho}) \cup \cN(T) \cup \{\rho\}.
        \end{equation}
     If $V$ is a finite set, the restriction to $V$ of  a partial order $\preceq$ coming from a 
     structure of rooted tree with vertex set \emph{containing} $V$  
     is called an {\bf arborescent partial order}. In this case, 
     $(V, \preceq)$ is called an {\bf arborescent poset}.
\end{definition}

It is immediate to see that for a rooted tree $T_{\rho}$, the root $\rho$ is the absolute minimum of 
$(T_{\rho}, \preceq_{\rho})$. In addition,  the partial order $\preceq_{\rho}$ has well-defined 
infima of pairs of points. 
This motivates (see Figure \ref{fig:ab-rho}):

    \begin{notation} \label{not:notinf}
        Let $T_{\rho}$ be a rooted tree. The {\bf infimum} of $a$ and $b$ for the partial order 
      $\preceq_{\rho}$, that is, the maximal element of $[ \rho a] \cap [ \rho  b]$, is denoted 
      $a\wedge_{\rho} b$. 
   \end{notation}
   
   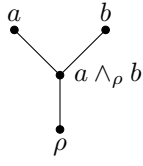
\begin{figure}
\centering
\begin{tikzpicture} [scale=0.6]
 \draw(0,0) -- (0,-1.2) ;  
 \draw(0,0) -- (1,1);
 \draw(0,0) -- (-1,1) ;
 
 \draw (0.1, 0) node[right]{$a \wedge_\rho b$} ;
 \draw (1, 1) node[above]{$b$} ;
\draw (-1,1) node[above]{$a$} ;
\draw (0,-1.2) node[below]{$\rho$} ;

 \node[draw,circle,inner sep=1pt,fill=black] at (0,0) {};
 \node[draw,circle,inner sep=1pt,fill=black] at (1,1) {};
\node[draw,circle,inner sep=1pt,fill=black] at (0,-1.2) {};
\node[draw,circle,inner sep=1pt,fill=black] at (-1,1) {};

    \end{tikzpicture}  
    \caption{The infimum of $a$ and $b$ for the partial order $\preceq_{\rho}$ 
         (see Notation \ref{not:notinf}).}  
    \label{fig:ab-rho}
    \end{figure}

It will be important in the sequel to distinguish the trees in which the root is an end 
(which implies, by Definition \ref{def:toptree}, that the tree has at least two vertices): 

\begin{definition} \label{def:endroot}
    An {\bf end-rooted tree} is a rooted tree $T_{\rho}$ whose root is an end.  
    Then the root $\rho$ has a unique direct   
    successor $\rho^+$ and each leaf 
    $a$ has a unique direct predecessor $a^{-}$. The {\bf core} $\overset{\circ}{T}_{\rho}$ 
    of the end-rooted tree $T_{\rho}$ 
    is the convex hull of $\{ \rho^+\} \cup \{ a^-, \:  a \in \cL(T_{\rho}) \}$, seen as 
    a tree rooted at $\rho^+$. 
    
    A rooted tree which is not end-rooted, that is, such that the root is interior, 
    is called {\bf interior-rooted}. 
\end{definition}

Given an arbitrary rooted tree, there is a canonical way to embed it in an 
end-rooted tree:

\begin{definition}  \label{def:canext}
   Let $T_{\rho}$ be a tree rooted at $\rho$. 
   Its {\bf extension} $\hat{T}_{\hat{\rho}}$ is the tree obtained from 
   $T_{\rho}$ by adding a new root $\hat{\rho}$, which is joined by an edge to $\rho$. 
\end{definition}

We defined arborescent posets starting from arbitrary rooted trees. Conversely, any 
arborescent poset $(V, \preceq)$ has a canonically associated rooted tree $T(V, \preceq)$ 
endowed with an underlying combinatorial tree. The root $\rho$ may not belong to $V$, 
but the vertex set of $T(V, \preceq)$ is exactly $V \cup \{\rho\}$. 
More precisely:

\begin{definition} \label{def:rt} 
   Let $(V, \preceq)$ be an arborescent poset. Its {\bf associated rooted tree} \linebreak
   $T(V, \preceq)$ is defined by: 
   \begin{itemize}
      \item  If $(V, \preceq)$ has a unique minimal element, then the root coincides 
          with it, and the edges are exactly the sets of the form $\{u, v\}$, where 
          $v$ is a direct predecessor of $u$. That is,  $T(V, \preceq)$ is the underlying tree 
          of the Hasse diagram of the poset. 
      \item If $(V, \preceq)$ has several minimal elements, then one considers 
         a new set $\hat{V}:= V \: \sqcup \: \{m\}$ and one extends the order $\preceq$ to it 
         by imposing that $m$ is a predecessor of all the elements of $V$. Then 
         one proceeds as in the previous case, working with $(\hat{V}, \preceq)$ 
         instead of  $(V, \preceq)$. In particular, the root is  the new vertex $m$. 
   \end{itemize}
     The {\bf extended rooted tree $\hat{T}(V, \preceq)$ 
     of} $(V, \preceq)$ is the extension of $T(V, \preceq)$ according to 
     Definition \ref{def:canext}. 
\end{definition}

The fact that the objects introduced in Definition \ref{def:rt} are always trees is a consequence 
of the following elementary proposition, whose proof we leave to the reader: 

\begin{proposition}  \label{prop:chararb}
   A partial order $\preceq$ on the finite set $V$ is arborescent if and only if any 
   element of it has at most one direct predecessor. 
\end{proposition}

The notion of \emph{extended rooted tree} of an arborescent poset  
will play an important role  in our context (see Remark \ref{rem:rootend}).
\begin{remark}
    We took the name of \emph{arborescent partial order} from \cite{FMS 14}. 
   The characterization given in Proposition \ref{prop:chararb} 
    was chosen in \cite{FMS 14} as the definition of this notion.
\end{remark}

One has the following fact, whose proof we leave to the reader:

\begin{proposition} \label{prop:recons}
   Let $T_{\rho}$ be a rooted tree. Then the rooted tree $T(\cV_{top}(T_{\rho}), \preceq_{\rho})$ 
   associated to the arborescent poset $(\cV_{top}(T_{\rho}), \preceq_{\rho})$ is isomorphic 
   to $T_{\rho}$ by an isomorphism which fixes $\cV_{top}(T_{\rho})$. 
\end{proposition}

      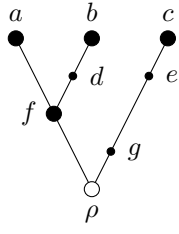
\begin{figure}
\centering
\begin{tikzpicture}  [scale=0.5]
 \draw(0,0) -- (-1, 2) -- (-2, 4);  
 \draw(0,0) -- (0.5 , 1)  -- (1.5 , 3) -- (2, 4);
 \draw(-1 ,  2) -- (-0.5 ,3)  -- (0 , 4) ;

 \draw (-2 , 4.2) node[above]{$a$} ;
 \draw (0 , 4.2) node[above]{$b$} ;
 \draw (2, 4.2) node[above]{$c$} ;
\draw (-0.3 , 3) node[right]{$d$} ;
\draw ( 1.7 , 3) node[right]{$e$} ;
\draw (-1.2 , 2) node[left]{$f$} ;
\draw ( 0.7 , 1) node[right]{$g$} ;
\draw (0 , -0.2) node[below]{$\rho$} ;

 \node[draw,circle,inner sep=2pt,fill=white] at (0,0) {};
 \node[draw,circle,inner sep=2pt,fill=black] at (-1 ,2 ) {};
\node[draw,circle,inner sep=2pt,fill=black] at ( -2, 4) {};
\node[draw,circle,inner sep=2pt,fill=black] at (0,4) {};
\node[draw,circle,inner sep=2pt,fill=black] at (2,4) {};
\node[draw,circle,inner sep=1pt,fill=black] at (-0.5 ,3) {};
\node[draw,circle,inner sep=1pt,fill=black] at (0.5 ,1) {};
\node[draw,circle,inner sep=1pt,fill=black] at ( 1.5 , 3) {};

    \end{tikzpicture}  
    \caption{The vertex set versus the topological vertex set in Example \ref{ex:topvert}.}  
    \label{fig:topvert}
    \end{figure}

A rooted tree may also be encoded as a supplementary structure 
on its set of leaves. Namely:

\begin{definition} \label{def:cluster}
    Let $T_{\rho}$ be a rooted tree. 
    To each point $v \in T_{\rho}$, associate its {\bf cluster} $\bK_{\rho}(v)$ 
    as the set of leaves which have $v$ as predecessor:
       \[ \bK_{\rho}(v) := \{ u \in \cL(T_{\rho}), \:  v \preceq_{\rho} u \}. \]
\end{definition}

   \begin{figure}
\centering
\begin{tikzpicture}  [scale=0.5]
 \draw(0,0) -- (-1, 1) -- (-2, 2) -- (-3,3);  
 \draw(0,0) -- (1 , 1)  -- (2 , 2) ;
 
 \draw(8,0) -- (7, 1) -- (6, 2) -- (5,3);  
 \draw(8,0) -- (9 , 1)  -- (10 , 2) ;
 
 \draw (8,0) -- (8, -1); 
   
 \draw (0.2 , 0) node[right]{$m$} ;
 \draw (-1.2 , 1) node[left]{$f$} ;
 \draw (-2.2, 2) node[left]{$d$} ;
\draw (-3.2 , 3) node[left]{$b$} ;
\draw ( 1.2 , 1) node[right]{$e$} ;
\draw (2.2 , 2) node[right]{$c$} ;

   \draw (8.2 , -1) node[right]{$\hat{m}$} ;
 \draw (8.2 , 0) node[right]{$m$} ;
 \draw (6.8 , 1) node[left]{$f$} ;
 \draw (5.8, 2) node[left]{$d$} ;
\draw (4.8 , 3) node[left]{$b$} ;
\draw ( 9.2 , 1) node[right]{$e$} ;
\draw (10.2 , 2) node[right]{$c$} ;

\draw (-1 , -2) node[right]{$T(V, \preceq)$} ;
\draw (7 , -2) node[right]{$\hat{T}(V, \preceq)$} ;

 \node[draw,circle,inner sep=2pt,fill=black] at (-1 ,1 ) {};
\node[draw,circle,inner sep=2pt,fill=black] at ( -2, 2) {};
\node[draw,circle,inner sep=2pt,fill=black] at (-3,3) {};
\node[draw,circle,inner sep=2pt,fill=black] at (1,1) {};
\node[draw,circle,inner sep=2pt,fill=black] at ( 2, 2) {};

 \node[draw,circle,inner sep=2pt,fill=black] at (7 ,1 ) {};
\node[draw,circle,inner sep=2pt,fill=black] at ( 6, 2) {};
\node[draw,circle,inner sep=2pt,fill=black] at (5,3) {};
\node[draw,circle,inner sep=2pt,fill=black] at (9,1) {};
\node[draw,circle,inner sep=2pt,fill=black] at ( 10, 2) {};

\node[draw,circle,inner sep=1pt,fill=black] at (0 ,0) {};
\node[draw,circle,inner sep=1pt,fill=black] at (8 ,0) {};
\node[draw,circle,inner sep=1pt,fill=black] at ( 8 , -1) {};

    \end{tikzpicture}  
    \caption{The two trees associated canonically to the arborescent poset $V$ of 
         Example \ref{ex:topvert}.} 
    \label{fig:twotrees}
    \end{figure}
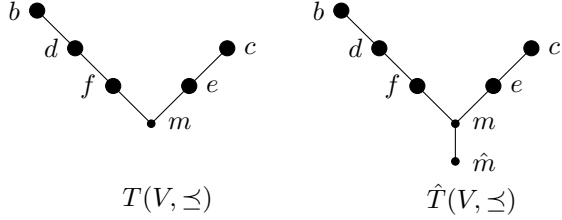

\begin{example} \label{ex:topvert}
  Figure \ref{fig:topvert} shows  a rooted tree $T$ with 
  vertex set $\{a, ..., g \}$ indicated by black bullets. It is 
  rooted at a point $\rho$ which is not a vertex, indicated by a white bullet. Note 
  that $\rho$ lies in the interior of the edge $[fg]$.  The  
  topological vertex set of $T$, indicated with bigger bullets, is 
  $\cV_{top}(T_{\rho}) = \{a,b,c,f,\rho\}$.  The arborescent poset 
  $(V = \{ b, c, d, e, f \}, \preceq_{\rho}) $ 
  (which is taken distinct from the vertex set,
   see Definition \ref{def:rootree}), may be described 
  by the strict inequalities, $f \prec_{\rho} d \prec_{\rho} b$ and $ e \prec_{\rho} c$,
  between directly comparable elements of $V$. 
   Notice that the poset  $V$ has two minimal 
   elements $e$ and $f$. 
   The root $m$ of the tree $T(V, \preceq)$ is a new point, 
   since we are in the second case of Definition \ref{def:rt}.
   The two rooted trees $T(V, \preceq)$ and $\hat{T}(V, \preceq)$ are drawn 
   in Figure \ref{fig:twotrees}. In both cases, the vertices corresponding to the elements 
   of $V$ are represented with bigger bullets. 
  The clusters associated to the vertices and the root of $T_\rho$ are: 
    \[   \begin{array}{c}
         \bK_{\rho}(a) = \{ a \}, \, \, \bK_{\rho}(b) =  \bK_{\rho}(d) = \{b \}, \, \,   
            \bK_{\rho}(f)  = \{ a, b \} ,   \\
         \, \,   \bK_{\rho}(c)  =  \bK_{\rho}(e) =  \bK_{\rho}(g)  = \{ c \}, \, \,  
             \bK_{\rho}(\rho) = \{a, b, c \}.   
               \end{array}
    \]
\end{example}

One has the following direct consequences of Definition \ref{def:cluster}:

\begin{proposition}  \label{prop:cluster}
    $\:  $
    
    \begin{enumerate}
          \item The cluster of a leaf $u$ is $\{u\}$ and the cluster of the root is the entire 
             set of leaves $\cL(T_{\rho})$. 
         \item The clusters $\bK_{\rho}(u)$ and $\bK_{\rho}(v)$ are disjoint if and only if 
            $u$ and $v$ are  incomparable. 
         \item One has $\bK_{\rho}(v) \subseteq \bK_{\rho}(u)$ if and only if $u$ is a predecessor 
             of $v$. 
         \item Two points $u, v \in T$ have the same cluster if and only if one is a predecessor 
            of the other one and the geodesic $[uv]$ does not contain nodes of $T_{\rho}$. 
    \end{enumerate}
\end{proposition}

Denote by $2^W$ the {\em power set} of a set $W$, that is, its set of subsets.  
As an immediate consequence of Proposition \ref{prop:cluster} one has 
(recall that $\cV_{top}(T_{\rho})$ denotes the \emph{topological 
vertex set} of $T_{\rho}$, introduced in Definition \ref{def:rootree}):

\begin{corollary}   \label{coro:ordinv}
    The \textbf{cluster map}:
         $$ \begin{array}{cccc}
                    \bK_{\rho}: &      \cV_{top}(T_{\rho})  &  \to & 2^{\cL(T_{\rho})} 
                    \\
                     &  u &  \to &  \bK_{\rho}(u)
               \end{array}  $$
       is decreasing from the poset $(\cV_{top}(T_{\rho}), \preceq_{\rho})$ 
       to the poset $(2^{\cL(T_{\rho})}, \subseteq)$. Moreover: 
          \begin{enumerate}
              \item  \emph{If $T_{\rho}$ is not end-rooted}, then $\bK_{\rho}$ is injective. 
              \item \emph{If $T_{\rho}$ is end-rooted}, then $\bK_{\rho}$ is injective in restriction 
                  to $\cV_{top}(T_{\rho})  \setminus \{ \rho \}$ and 
                  $\bK_{\rho}(\rho) = \bK_{\rho}(\rho^+)$, 
                  where $\rho^+$ is the unique direct successor of $\rho$ in the poset 
                  $(\cV_{top}(T_{\rho}), \preceq_{\rho})$. 
          \end{enumerate}
\end{corollary}

Proposition \ref{prop:cluster} may be reformulated 
by saying that the image of the cluster map is a \emph{hierarchy} in the following sense: 

\begin{definition}  \label{def:hierarchy}
   A {\bf hierarchy} on the finite set $X$ is a subset  of $2^X\setminus \{\emptyset\}$ 
   (whose elements are called {\bf clusters}) satisfying the following 
   properties:
       \begin{enumerate}
           \item All the one-element subsets of $X$ as 
              well as $X$ itself are clusters.  
           \item Given any two clusters, they are either disjoint or one is included inside the other. 
       \end{enumerate}
\end{definition}

 In fact, those properties characterize 
completely the images of cluster maps associated to rooted trees with given 
leaf set (folklore, see \cite[Introduction]{BD 98}):

\begin{proposition}   \label{prop:charimage}
   The images of the cluster maps associated to the rooted trees with finite 
   sets of leaves $X$ are exactly the {\bf hierarchies} on $X$. 
\end{proposition}

Now, if one orders a hierarchy by \emph{reverse inclusion}, one gets an arborescent poset: 

\begin{lemma}  \label{lem:revincl}
    Let $\cH$ be a hierarchy on the finite set $X$. Define the partial order $\preceq_{ri}$ 
    on $\cH$ by:
        $$ K_1 \preceq_{ri} K_2 \:  \Longleftrightarrow  \:   K_1 \supseteq K_2.$$
     Then $(\cH, \preceq_{ri})$ is an arborescent poset. 
\end{lemma}

\begin{proof}
    By Proposition \ref{prop:chararb}, it is enough to check that for any cluster 
    $K_1  \ne X$, there 
    exists a unique cluster $K_2$ strictly containing it  and such that there are no other 
    clusters between $K_1$ and $K_2$. But this comes from the fact that, by condition 
    (2) of Definition \ref{def:hierarchy}, all the clusters containing $K_1$ form a chain 
    (that is, a totally ordered set) under inclusion. 
\end{proof}

Conversely, one has the following characterization of 
 arborescent posets coming from hierarchies, whose proof we leave to the reader: 

\begin{proposition}  \label{prop:charhier}
   An arborescent poset $(V, \preceq)$ is isomorphic to the poset defined by a hierarchy if 
   and only if any non-maximal element has at least two direct successors. 
   In particular, the associated rooted tree $T(V, \preceq)$ is never end-rooted.
\end{proposition}

\subsection{Ultrametric spaces and dated rooted trees}  \label{ultramdepth}
$\:$ 
\medskip

In this subsection we explain the relation between \emph{finite ultrametric spaces} 
and rooted trees endowed with a \emph{depth function}. 
\medskip

Let us first fix our notations and vocabulary about metric spaces: 

\begin{notation}  \label{not:ball}  
  If $(X,d)$ is a metric space, then a {\bf closed ball} of it is a subset of the form:
     $\B(a,  r):= \{ p \in X\: | \:  d(a, p) \leq r \},$
  where $a$ denotes any point of $X$ and $r \in [0, + \infty)$. 
  Each time a subset of $X$ is presented in this way, one says that $a$ is {\bf a center} 
  and $r$ is {\bf a radius} of it. 
  The {\bf diameter} of a subset $Y \subset X$ is   
  $\mbox{diam}(Y) := \sup\{ d(x,y), x,y \in Y\} \in [0, + \infty].$
\end{notation}

In Euclidean geometry, a closed ball has a unique center and a unique radius. 
None of those two properties remains true in general. There is an extreme situation 
in which \emph{any} point of a closed ball is a center of it:

\begin{definition} \label{def:ultram}
    Let $(X, d)$ be a metric space. The distance function $d: X \times X \to [0, \infty)$ 
    is called an {\bf ultrametric} if one has the following strong form of the triangle 
    inequality, called the {\bf ultrametric inequality}:
       $$ d(a,b) \leq \max \{ d(a, c), d(b,c) \}, \:  \mbox{ for all }  \:   a, b, c \in X.$$
    In this case, we say that $(X,d)$ is {\bf an ultrametric space}. 
\end{definition}

One has the following characterizations of ultrametricity, which result immediately  
from the definition: 

\begin{proposition}   \label{prop:equivultra}
   Let $(X, d)$ be a metric space. Then the following conditions are equivalent:
      \begin{enumerate}
         \item \label{first} 
              $d$ is an ultrametric. 
         \item  \label{second} 
              All the triangles are either equilateral or 
             isosceles with the unequal side being the shortest. 
         \item \label{third} 
               For any closed ball, all its points are centers of it. 
         \item \label{clusterchar} 
             Given any two closed balls, they are either disjoint, or one of them 
           is contained into the other one. 
      \end{enumerate}
\end{proposition}

As a consequence of Proposition \ref{prop:equivultra}, we have the following property:

\begin{lemma}  \label{lem:radiusultra}
    Let $(X, U)$ be a finite ultrametric space. If $\cD$ is a closed ball, then 
    its diameter  $\mathrm{diam} (\cD)$ 
    is the minimal radius $r$ such that $\cD = \B(a,r)$ for any $a \in \cD$. 
\end{lemma}

The prototypical examples of ultrametric spaces are the fields of $p$-adic numbers or, more 
generally, all the fields endowed with a non-Archimedean norm.  
Our goal in the rest of this section is to describe  the canonical presentation of  
finite ultrametric spaces  as {\em sets of leaves 
of finite rooted trees} (see Proposition \ref{prop:corrultradepth} below).  
A pleasant elementary 
introduction to this view-point is contained in Holly's paper \cite{H 01}. 

The basic fact indicating that rooted trees are related with ultrametric spaces is the 
similarity of the conditions defining hierarchies (see Definition \ref{def:hierarchy}) with the 
characterization of ultrametrics given as point 
(\ref{clusterchar}) of Proposition \ref{prop:equivultra}. 
This characterization, combined with the fact that the closed balls of radius $0$ are 
exactly the subsets with one element, and the fact that on a finite metric space, 
the closed balls of sufficiently big radius are the whole set, shows that:

\begin{lemma}  \label{lem:ballhierarchy}
  The set $\cB(X,U)$ of closed balls  of a finite ultrametric space $(X,U)$  
  is a hierarchy on $X$. 
\end{lemma}

 As a consequence of Lemmas \ref{lem:revincl} and \ref{lem:ballhierarchy}, 
 an ultrametric $U$ on a finite set $X$ defines canonically two rooted trees with leaf-set $X$ 
 and topological vertex set $\cB(X,U)$: 
  
  \begin{definition}  \label{def:asstree}
      Let $U$ be an ultrametric on the finite set $X$. The {\bf interior-rooted tree $T^U$ 
      associated to the ultrametric} $U$ is the rooted tree $T(\cB(X,U), \preceq_{ri})$ 
      determined by the arborescent poset of  closed balls of $U$.  
      The  {\bf end-rooted tree $\hat{T}^U$ associated to the ultrametric} $U$ 
      is the extended rooted tree $\hat{T}(\cB(X,U), \preceq_{ri})$. 
  \end{definition}
  The terminology is motivated by Proposition \ref{prop:charhier}, which implies that 
  $T^U$ is indeed always interior-rooted and $\hat{T}^U$ always end-rooted. 
 By Definition \ref{def:rt}  and Lemma \ref{lem:revincl},  
 the root $\rho$ of $T^U$ corresponds to the set $X$,
  while the root $\hat{\rho}$ of $\hat{T}^U$ is defined as the immediate predecessor of $\rho$ in  
  the tree $\hat{T}^U$.

One may encode also the \emph{values} of the metric $U$ 
on its end-rooted tree $\hat{T}^U$, as a decoration 
on its set $\cI(\hat{T}^U)$ of interior vertices:

\begin{definition}  \label{def:depthultra}
      Let $(X,U)$ be a finite ultrametric space. Its {\bf diametral function}: 
            $$\delta^U : \cI(\hat{T}^U)    \to \R_+^*,$$ 
       defined on the set of interior vertices of the end-rooted 
       tree $\hat{T}^U$ of  $(X,U)$, 
      associates to each vertex $u \in \cI(\hat{T}^U) $ 
      the diameter of its cluster $\bK_{\hat{\rho}}(u) $.
\end{definition}

The root $\rho$ of $T^U$ is always an interior vertex of $\hat{T}^U$, and its diameter  
$\delta^U (\rho)$ is equal to the diameter of $X$. Notice also that  
$\bK_{\rho}(u) = \bK_{\hat{\rho}}(u) $, for all $u \in  \cI(\hat{T}^U)  = \cI({T}^U) $.

The diametral function of a finite ultrametric space is a \emph{depth function on the 
associated end-rooted tree} in the following sense:

\begin{definition} \label{def:depthmap}
     A {\bf depth function} on a rooted tree $T_{\rho}$ is a strictly decreasing function: 
         $$\delta: (\cI(T_{\rho}) , \preceq_{\rho})  \to (\R_+^* , \leq).$$ 
     That is,  ${\delta}(v) <  {\delta}(u)$ whenever $u \prec_{\rho} v$. 
     A pair $(T_{\rho}, {\delta})$ of a rooted tree and a depth function on it is called 
     a {\bf depth-dated tree}. 
\end{definition}

Intuitively, such a function $\delta$ measures the depth of the interior vertices as seen 
from the leaves, 
if one imagines that the leaves are above the root, as modeled by the partial order 
$\preceq_{\rho}$.

We have  explained  
how to pass from an ultrametric on a finite set to a 
depth-dated rooted tree (see Definitions \ref{def:asstree} and \ref{def:depthultra}). 
Conversely, given such a tree, 
one may construct an ultrametric on its set of  leaves (recall  that $a \wedge_{\rho} b$ 
is the infimum of $a$ and $b$ 
for the partial order $\preceq_{\rho}$, see Notation  
\ref{not:notinf}):

\begin{lemma}  \label{lem:depthultra}
   Let $(T_{\rho}, \delta)$ be a depth-dated rooted tree. Then the function 
   $U^{\delta}: \cL(T_{\rho}) \times \cL(T_{\rho}) \to \R_+$ defined by:
      $$ U^{\delta}(a,b) := \left\{ \begin{array}{ll}
                                               \delta(a \wedge_{\rho} b) & \mbox{ if } a \neq b, \\
                                               0 & \mbox{ if } a = b,
                                           \end{array}
                                                   \right.  $$
   is an ultrametric on $\cL(T_{\rho})$. 
 \end{lemma} 
 
\begin{proof}
   Consider $a,b,c \in  \cL(T_{\rho})$. The inequality: 
       $$  U^{\delta}(a,b) \leq \max \{ U^{\delta}(a,c), U^{\delta}(b,c) \}$$ 
   is equivalent to:
      $$ {\delta}(a \wedge_{\rho} b)  
            \leq \max \{  {\delta}(a \wedge_{\rho} c),  {\delta}(b \wedge_{\rho} c) \},$$
    which is in turn a consequence of the facts that ${\delta}$ is a depth function and that: 
     $$a \wedge_{\rho} b  \geq \min \{  a \wedge_{\rho} c,  b \wedge_{\rho} c \}.$$
     The previous inequality, including the existence of this minimum (taken relative 
     to the rooted tree partial order $\prec_{\rho}$)  
     is a basic property of rooted trees. 
 \end{proof}

 Moreover, if the tree $T_{\rho}$ is end-rooted, then one may reconstruct it as the end-rooted tree 
 associated to the ultrametric space $(\cL(T_{\rho}), U^{\delta})$: 
 
 \begin{proposition} \label{prop:reconstr} 
   Let $(T_{\rho}, {\delta})$ be a depth-dated end-rooted tree.  There exists a unique 
   isomorphism fixing the set of leaves between the combinatorial rooted trees 
   underlying the depth-dated trees:
     \begin{itemize}
          \item $T_{\rho}$ endowed with the topological vertex set $\cV_{top}(T_{\rho})$ 
              and with the restriction to its set of interior vertices of the depth function ${\delta}$;  
          \item the depth-dated tree $(\hat{T}^{U^{\delta}}, \delta^{U^{\delta}})$ associated to the 
          ultrametric $U^{\delta}$.
     \end{itemize} 
\end{proposition}

Taken together, the previous considerations prove the announced bijective 
correspondence between ultrametrics on a finite set $X$ and a special type 
of depth-dated end-rooted trees with set of leaves $X$: 

\begin{proposition}  \label{prop:corrultradepth}
        Let $X$ be a finite set. The map which associates to an ultrametric $U$ on $X$ 
        the diametral function $\delta^U$ on the end-rooted tree $\hat{T}^U$ realizes a bijective 
        correspondence between ultrametrics on $X$ and isomorphism classes 
        of depth-dated end-rooted rooted trees $(T_{\rho}, {\delta})$ with set of vertices equal to 
        their topological vertex set and with set of leaves equal to $X$. 
\end{proposition}

The following notion is dual to that of  \emph{depth functions}:

\begin{definition} \label{def:heightmap}
     A {\bf height function} on a rooted tree $T_{\rho}$ is a strictly increasing function: 
         $$h: (\cI(T_{\rho}) , \preceq_{\rho})  \to (\R_+, \leq).$$ 
     That is,   $h(u) < h(v)$ whenever $u \prec_{\rho} v$. 
     A pair $(T_{\rho}, h)$ of a rooted tree and a height function on it is called 
     a {\bf height-dated tree}. 
\end{definition}

\begin{remark} \label{lem:dissym}
Note the slight assymmetry of the two definitions: we impose that depth functions 
take positive values, but we allow a height function to vanish.
This assymmetry is motivated by the fact that we use 
depth functions to define ultrametrics by Lemma \ref{lem:depthultra}. The condition 
of strict increase on a height function imposes that a vanishing may occur only at the minimal 
element of $\cI(T_{\rho})$, which is either the root $\rho$  (if $T_{\rho}$ is interior-rooted) or 
its immediate successor in $\cV(T_{\rho})$ (if $T_{\rho}$ is end-rooted). 
\end{remark}

Any strictly decreasing function allows to transform height functions into depth functions: 

\begin{lemma} \label{lem:transf}
    Any strictly decreasing map:
       $$s: (\R_+, \leq) \to (\R_+^*, \leq)$$
     transforms by left-composition all height functions on a rooted tree into depth functions. 
\end{lemma}

\begin{remark}  \label{rem:inspir}
In \cite{BD 98}, B\"ocker and Dress defined more general \emph{symbolically dating maps} 
on trees, taking values in arbitrary sets, and characterize the associated 
\emph{symbolic ultrametrics} by a list of axioms. We don't use here that generalized 
setting. Nevertheless we mention it because that paper inspired us in our work. 
For instance, we introduced the names \emph{depth-dated/height-dated tree} 
by following its ``\emph{dating}'' terminology (which seems standard in part of the 
mathematical literature concerned with problems of classifications, as 
mathematical phylogenetics). 
\end{remark}

\begin{example}  \label{ex:constree}
    Consider a set $X = \{u, v, x, y, z \}$ and a function $U : X^2 \to \R_+$ whose 
    matrix $(U(a,b))_{a,b \in X}$ is (for the order $u < \cdots < z$): 
           \[
\begin{pmatrix}
0 & 1 &  1 & 2 & 3  \\
1 & 0 & 1 & 2 &  3 \\
1 & 1  &  0  & 2 &  3  \\
2 &  2 &  2 &  0 &   3 \\
3 &  3  & 3 &   3 &  0 
\end{pmatrix}. 
\]
It is immediate to check that $U$ is an ultrametric distance. The hierarchy of its closed 
balls is:
  $$\left\{ \:  \{u\}, ..., \{z\}, \{u,v,x\}, \{u,v,x,y\}, \{u,v,x,y,z\} \:  \right\}. $$ 
 The associated rooted trees $T^U$ and $\hat{T}^U$ 
 are represented in Figure \ref{fig:assoctree}, $T^U$ being drawn with thicker segments. 
 Near each vertex is written the associated cluster. Near each interior vertex $u$ of 
  $\hat{T}^U$ is also written the value of the diametral function $\delta^U(u)$, that is, 
 the diameter of the cluster $\bK_{\rho}(u)$. 
 
   \begin{figure}
\centering
\begin{tikzpicture}  [scale=0.8]

 \draw(0,0) -- (0,1) ;  
 \draw [very thick] (0,1) -- (-1,2);
 \draw  [very thick] (0,1) -- (3,4) ;
  \draw [very thick] (-1,2) -- (-2,3) ; 
  \draw [very thick] (-1,2) -- (1,4) ;  
 \draw [very thick] (-2,3) -- (-3,4);
 \draw [very thick] (-2,3) -- (-2,4) ;
  \draw [very thick] (-2,3) -- (-1,4) ; 
   
 \draw (-3, 4.3) node[above]{$\{u\}$} ;
 \draw (-2, 4.3) node[above]{$\{v\}$} ;
 \draw (-1, 4.3) node[above]{$\{x\}$} ;
\draw (1, 4.3) node[above]{$\{y\}$} ;
\draw (3, 4.3) node[above]{$\{z\}$} ;
\draw (0, -0.2) node[below]{$\hat{\rho}$} ;

 \draw (-2.2, 3) node[left]{$\{u, v, x \}$} ;
 \draw (-1.2, 2) node[left]{$\{u, v, x, y \}$} ;
 \draw (-0.2, 1) node[left]{$ \rho = \{u, v, x, y, z \}$} ;
 
  \draw (-1.8, 3) node[right]{$1$} ;
 \draw (-0.8, 2) node[right]{$2$} ;
 \draw (0.2, 1) node[right]{$3$} ;

 \node[draw,circle,inner sep=1pt,fill=black] at (0,0) {};
 \node[draw,circle,inner sep=1pt,fill=black] at (0,1) {};
\node[draw,circle,inner sep=1pt,fill=black] at (-1,2) {};
\node[draw,circle,inner sep=1pt,fill=black] at (-2,3) {};
\node[draw,circle,inner sep=1pt,fill=black] at (-3,4) {};
\node[draw,circle,inner sep=1pt,fill=black] at (-2,4) {};
\node[draw,circle,inner sep=1pt,fill=black] at (-1,4) {};
\node[draw,circle,inner sep=1pt,fill=black] at (1,4) {};
\node[draw,circle,inner sep=1pt,fill=black] at (3,4) {};

    \end{tikzpicture}  
    \caption{The depth-dated tree associated to the ultrametric of Example \ref{ex:constree}.}  
    \label{fig:assoctree}
    \end{figure}
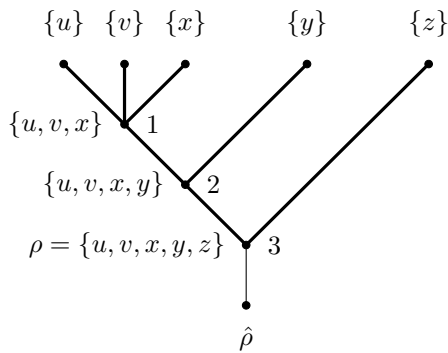   
 
\end{example}

\subsection{Additive distances on trees}  \label{addist}
$\:$  \medskip

Let us pass now to the notion of \emph{additive distance} on an unrooted tree. 
Our aim in this subsection is to explain in which way such a distance defines plenty of 
ultrametrics (see Proposition \ref{prop:addistultram}).

\begin{definition}  \label{def:addmet}
    Let $(V, A)$ be  a combinatorial tree. An {\bf additive distance} on it is a 
    symmetric map $d: V \times V \to \R_+$ such that:
       \begin{enumerate}
           \item $d(u,v) =0$ if and only if $u =v$;  
           \item $d(u,v) + d(v,w) = d(u,w)$ whenever $v \in [uw]$.  
       \end{enumerate}
\end{definition}

  Of course, the additive distance $d$ is a metric on $V$. 
Buneman \cite{B 74} characterized such metrics in the following way:

\begin{proposition}  \label{prop:chardist}
   A metric $d$ on the finite set $V$ comes from an additive distance function on 
   a combinatorial tree with vertex set $V$ if and only if it satisfies the following 
   {\bf four points condition}:
         $$d(l,u) + d(v,w) \leq 
            \max \{ d(l,v) + d(u,w), d(l,w) + d(u,v)  \}  \mbox{ for any } l,u,v,w \in V. $$
\end{proposition}

    In fact, one has the following more precise inequality, which we leave as an exercise to the reader:
    
    \begin{proposition}  \label{prop:moreprec}
    Let  $d$ be an additive distance on the finite combinatorial tree $(V,A)$. 
       Consider also its geometric realization, inside which will be taken convex hulls. 
       Then, for every $l, u, v, w \in V$, one has:
           $$d(l,v) + d(u,w) \geq  d(l,u) +  d(v,w) \:  \Longleftrightarrow \:   [lv] \cap [uw] 
              \neq \emptyset,$$
        with equality if and only if one has moreover $[lu] \cap [vw] \neq \emptyset$, that is, 
        if and only if the convex hull $[luvw]$ is like in the right-most tree of Figure \ref{gentree}, 
        or as in one of its degenerations. 
    \end{proposition}
    
 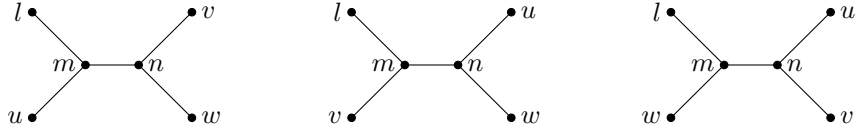
\begin{figure}
\centering
\begin{tikzpicture} [scale=0.7]
    \draw (0,0) -- (-1,1) ; 
    \draw (0,0) -- (-1,-1) ; 
    \draw (0,0) -- (1,0) ;
    \draw (1,0) -- (2,1) ; 
    \draw (1,0) -- (2,-1) ;
    \draw (-1,1) node[left]{$l$} ; 
     \draw (-1,-1) node[left]{$u$} ; 
     \draw (2,1) node[right]{$v$} ;
     \draw (2,-1) node[right]{$w$} ; 
     \draw (0,0) node[left]{$m$} ;
     \draw (1,0) node[right]{$n$} ; 
           \draw (6,0) -- (5,1) ; 
           \draw (6,0) -- (5,-1) ; 
           \draw (6,0) -- (7,0) ;
           \draw (7,0) -- (8,1) ; 
           \draw (7,0) -- (8,-1) ; 
            \draw (5,1) node[left]{$l$} ; 
            \draw (5,-1) node[left]{$v$} ; 
            \draw (8,1) node[right]{$u$} ;
            \draw (8,-1) node[right]{$w$} ; 
             \draw (6,0) node[left]{$m$} ;
             \draw (7,0) node[right]{$n$} ; 
    \draw (12,0) -- (11,1) ; 
    \draw (12,0) -- (11,-1) ; 
    \draw (12,0) -- (13,0) ;
    \draw (13,0) -- (14,1) ; 
    \draw (13,0) -- (14,-1) ;
    \draw (11,1) node[left]{$l$} ; 
     \draw (11,-1) node[left]{$w$} ; 
     \draw (14,1) node[right]{$u$} ;
     \draw (14,-1) node[right]{$v$} ; 
      \draw (12,0) node[left]{$m$} ;
     \draw (13,0) node[right]{$n$} ; 
     
      \node[draw,circle,inner sep=1pt,fill=black] at (0,0) {};
 \node[draw,circle,inner sep=1pt,fill=black] at (-1, -1) {};
\node[draw,circle,inner sep=1pt,fill=black] at (-1, 1) {};
\node[draw,circle,inner sep=1pt,fill=black] at (1, 0) {};
\node[draw,circle,inner sep=1pt,fill=black] at (2, 1) {};
\node[draw,circle,inner sep=1pt,fill=black] at (2, -1) {};

 \node[draw,circle,inner sep=1pt,fill=black] at (6,0) {};
 \node[draw,circle,inner sep=1pt,fill=black] at (5,1) {};
\node[draw,circle,inner sep=1pt,fill=black] at (5,-1) {};
\node[draw,circle,inner sep=1pt,fill=black] at (7, 0) {};
\node[draw,circle,inner sep=1pt,fill=black] at (8 , 1) {};
\node[draw,circle,inner sep=1pt,fill=black] at ( 8, -1) {};

 \node[draw,circle,inner sep=1pt,fill=black] at (12,0) {};
 \node[draw,circle,inner sep=1pt,fill=black] at (11,1) {};
\node[draw,circle,inner sep=1pt,fill=black] at (11 ,-1) {};
\node[draw,circle,inner sep=1pt,fill=black] at (13, 0) {};
\node[draw,circle,inner sep=1pt,fill=black] at (14, 1) {};
\node[draw,circle,inner sep=1pt,fill=black] at ( 14, -1) {};

    \end{tikzpicture}  
    \caption{The three possible generic trees with four leaves. } 
    \label{gentree}
\end{figure}

    More degenerate configurations are obtained 
      by contracting to points one or more segments of Figure \ref{gentree}. 
      For instance, it is possible to have $m=n$ or $m=l$, etc.

Assume now that $(T,d)$ is a topological tree endowed with an additive distance  
(on its underlying combinatorial tree). Choose a root $\rho \in \cV(T)$. The 
distance function may now be encoded alternatively as a more general \emph{height function}:

\begin{definition}  \label{def:heightbis}
  The {\bf remoteness function} associated to the additive distance  $d$ on the 
  rooted tree $T_{\rho}$ is defined by:
       $$ \begin{array}{cccc} 
                h_{d,\rho} :   &  \cV(T_{\rho})  & \to &  \R_+ \\
                          &  u & \to &  d(u, \rho).
             \end{array} $$
\end{definition}

Note that, as may be verified simply by looking at Figure \ref{fig:ab-rho}, the remoteness  
function allows to reconstruct the additive distance: 

\begin{lemma}  \label{lem:heightinf} 
   Assume that $d$ is an additive distance on the rooted tree $T_{\rho}$. 
   Then one has the following equivalent equalities: 
     \begin{enumerate}
        \item 
           $d(a,b) = h_{d,\rho}(a) + h_{d,\rho}(b) - 2 h_{d,\rho}(a\wedge_{\rho} b)$ 
             for all  $a,b \in \cV(T_{\rho})$. 
        \item  
              $h_{d,\rho}(a\wedge_{\rho} b) = \frac{1}{2} \cdot( d(\rho,a) + d(\rho,b) - d(a,b) )$ 
                for all $a,b \in \cV(T_{\rho})$. 
      \end{enumerate}
\end{lemma}

In the following results we consider an end-rooted tree $T_{\hat{\rho}}$ with non-empty core 
$\overset{\circ}{T}_{\rho}$ (see Definition \ref{def:endroot}). 
Notice that the set of vertices of the core $\overset{\circ}{T}_{\rho}$ is
the set of interior vertices of $T_{\hat{\rho}}$ and that 
$\overset{\circ}{T}_{\rho}$ is considered as a rooted-tree, 
the root $\rho$ being the immediate successor of $\hat{\rho}$ in $T_{\hat{\rho}}$.
It is immediate to see that:

\begin{lemma} \label{lem:heightcore}
    Consider an end-rooted tree $T_{\hat{\rho}}$  with non-empty core $\overset{\circ}{T}_{\rho}$. 
     Then the map:
        $ d \to h_{d, \rho},$
     which associates to an additive distance function of  $\overset{\circ}{T}_{\rho}$ 
     its remoteness function, establishes a bijection from the set of additive 
     distances on $\overset{\circ}{T}$ to the set of height functions of $T_{\hat{\rho}}$  
     which vanish  at $\rho$. 
 \end{lemma}

Combining Lemmas  \ref{lem:heightcore}, \ref{lem:transf} and \ref{lem:depthultra}, we get:

\begin{proposition} \label{prop:addistultram} 
     Assume that $T_{\hat{\rho}}$ is an end-rooted tree
    with non-empty core $\overset{\circ}{T}_{\rho}$, which is 
     endowed with an additive distance function $d$. 
    Then, for any strictly decreasing function  $s: (\R_+, \leq) \to (\R_+^*, \leq)$, the function  
    $U^{d, \rho, s} : \cL(T_{\hat{\rho}}) \times \cL(T_{\hat{\rho}}) \to \R_+$ defined by:
       $$ U^{d, \rho, s}(a,b) := \left\{ \begin{array}{ll}
                                               s(h_{d,\rho}(a\wedge_{\rho} b) ) & \mbox{ if } a \neq b \\
                                               0 & \mbox{otherwise},
                                           \end{array}
                                                   \right.  $$
   is an ultrametric on the set of leaves $\cL(T_{\hat{\rho}})$.
\end{proposition}

Therefore, starting from an end-rooted tree whose core is endowed with an additive distance, 
one gets 
plenty of ultrametric spaces, depending on the choice of map $s$. Each such 
ultrametric space has two associated rooted trees, as explained in Definition  
\ref{def:asstree}. By Proposition \ref{prop:corrultradepth}, the end-rooted one may be 
identified topologically with the initial end-rooted tree:

\begin{proposition}  \label{prop:embedtree}
     Let $T_{\hat{\rho}}$ be an end-rooted tree
    with non-empty core $\overset{\circ}{T}_{\rho}$. We assume that  $\overset{\circ}{T}_{\rho}$  
    is endowed with an additive distance function $d$. 
   Consider $T_{\hat{\rho}}$ as a combinatorial rooted tree with vertex set equal to 
   its topological vertex set
   $\cV_{top}(T_{\hat{\rho}})$, as defined by (\ref{eq:topvertex}). Let
   $s: (\R_+, \leq) \to (\R_+^*, \leq)$ be a strictly decreasing map 
   and let $U^{d, \rho, s}: \cL(T_{\hat{\rho}}) \times \cL(T_{\hat{\rho}}) \to \R_+$ 
   be the ultrametric defined in Proposition \ref{prop:addistultram}. Then: 
       \begin{enumerate}
             \item The end-rooted tree 
                  $\hat{T}^{U^{d, \rho, s}}$ is uniquely isomorphic with $T_{\hat{\rho}}$ as a 
                  combinatorial rooted tree 
                  with leaf-set $\cL(T_{\hat{\rho}})$. 
                   In restriction to $\cI(T_{\hat{\rho}})$, 
                   this isomorphism identifies  the diametral function $\delta^{U^{d, \rho, s}}$ 
                   of Definition \ref{def:depthultra} with $s \circ h_{d,\rho}$. 
            \item  The previous isomorphism identifies the interior-rooted tree  
                  $T^{U^{d, \rho, s}}$ with the convex hull $[\cL(T_{\hat{\rho}}) ]$, as rooted trees  
                   with leaf-set $\cL(T_{\hat{\rho}})$. 
       \end{enumerate}
\end{proposition}
Proposition \ref{prop:embedtree} is a special case of the theory explained by 
B\"ocker and Dress in \cite{BD 98}.

\section{Arborescent singularities and their ultrametric spaces of branches}
\label{sec:main}

In this section we introduce the notion of \emph{arborescent} singularity   
and we prove that, given any good resolution of such a singularity, there is a natural additive 
metric on its dual tree, constructed from a determinantal identity of Eisenbud 
and Neumann. We deduce from this additivity the announced results about the fact 
that the functions $U_L$ defined in the introduction are ultrametrics, and their 
relation with the dual trees of resolutions.

\subsection{Determinant products for arborescent singularities} \label{subsec:det}
$\: $ 
\medskip

In this subsection we explain the basic facts about arborescent singularities 
needed in the sequel. 
\medskip

\begin{definition}  \label{def:arb}
    A normal surface singularity is called {\bf arborescent} if the dual graph 
    of some good resolution of it is a tree. 
\end{definition}

It is immediate to see, using the fact that there exists a minimal good resolution,  
which any good resolution dominates by a sequence of blow-ups of points, that 
the dual graph of some good resolution is a tree if and only if this is the case for any 
good resolution. 

\begin{remark}  \label{rem:exarb}
 All normal quasi-homogeneous, rational, minimally elliptic singularities which are 
not cusp singularities or Neumann and Wahl's \cite{NW 05} 
splice quotient singularities are arborescent. Note that the first three classes of 
singularities are special cases of splice quotients whenever their boundaries 
are rational homology spheres (which is always the case for rational ones, but 
means that the quotient by the $\C^*$-action is a rational curve in the quasihomogeneous case, 
and that one does not have a simply elliptic singularity -- a special case of 
quasihomogeneous singularities, in which this quotient is elliptic and there are no special 
orbits -- in the minimally elliptic case). This was proved by Neumann \cite{N 83} for 
the quasi-homogeneous singularities and by Okuma \cite{O 06} for the other classes 
(see also \cite[Appendix]{NW 05}). 
Note that all splice quotients  are special 
cases of normal surface singularities with rational homology sphere links, which 
in turn are all arborescent by \cite[Proposition 3.4 of Chap. 2]{D 92}. 
\end{remark}

\begin{remark}   \label{rem:frompepe}
    Jos\'e Seade told us that arborescent singularities had also appeared, but 
    without receiving a special name, in Camacho's paper \cite{C 88}. 
\end{remark}

The notions explained in the following definition were introduced with slightly different 
terminology by Eisenbud and Neumann \cite[Sect. 20, 21]{EN 85} (recall that the 
notion of \emph{determinant} of a weighted graph was explained in Definition \ref{def:detgraph}): 

\begin{definition} \label{def:detrees}
   Let $S$ be an arborescent singularity. Consider any good resolution of it. 
   For each vertex $u$ of its dual tree $\Gamma$ and each edge $e$ 
   containing $u$, we say that the {\bf subtree $\Gamma_{u,e}$ of $u$ 
   in the direction $e$} is the full subtree of $\Gamma$ whose vertices are 
   those vertices $t$ of $\Gamma$ distinct from $u$, which are seen from 
   $u$ in the direction of the edge $e$, that is, such that 
   $e \subset [ut]$ (see Figure \ref{fig:subgraph}). 
   The {\bf edge determinant $\det_{u,e}(\Gamma)$ 
   at the vertex  $u$ in the direction $e$} is the determinant of $\Gamma_{u,e}$. 
   For any $v, w \in \cV$, define  the {\bf determinant product} $p(v, w)$ 
   {\bf of the pair} $(v, w)$ as the product of the determinants at all the points of the 
   geodesic $[vw]$ which connects $v$ and $w$, in the directions of the edges 
   which are {\em not} contained in that geodesic. 
\end{definition}

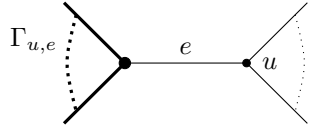
\begin{figure}
\centering
\begin{tikzpicture} [scale=0.8]

    \draw[very thick](0,0) -- (-1,1) ; 
       
 \draw[very thick,dotted] (-0.8, 0.8) .. controls (-1 ,0) .. (-0.8, -0.8);
 \draw[dotted] (2.8, 0.8) .. controls (3 ,0) .. (2.8, -0.8);

    \draw[very thick] (0,0) -- (-1,-1) ; 
     \draw(0,0) -- (2,0) ;
    \draw (2,0) -- (3,1) ; 
    \draw (2,0) -- (3,-1) ;
        \draw (2.1,0) node[right]{$u$} ;
       \draw (1,0) node[above]{$e$} ;
         \draw (-1.5,0) node[above]{$\Gamma_{u,e}$} ;
        \node[draw,circle,inner sep=1.5pt,fill=black] at (0,0) {};
 \node[draw,circle,inner sep=1pt,fill=black] at (2,0) {};
    \end{tikzpicture}  
    \caption{The subtree $\Gamma_{u,e}$ of Definition \ref{def:detrees} 
          is sketched with thicker segments. }
    \label{fig:subgraph}
    \end{figure}

Note that the definition implies that  $p(v, w) = p(w, v) \in \N^*$ for any $v,w \in \cV$.

In order to compute  determinant products  
in concrete examples, it is important to be able to compute rapidly determinants of weighted 
trees. One could use the general algorithms of linear algebra. Happily, there exists 
a special algorithm adapted to tree determinants, which was presented in Duchon's 
thesis  \cite[Sect. III.1]{D 82} and studied in \cite[Section 21]{EN 85}. We used it 
a lot for our experimentations. 
This algorithm may be formulated as follows:

\begin{proposition} \label{prop:duchon}
   Let $\Gamma_u$ be the dual tree of a good resolution of an arborescent singularity, 
   rooted at one of its vertices $u$. For any vertex $v$ of $\Gamma_u$, 
   denote by $-\alpha_v <0$ the self-intersection 
   of the component $E_v$.  Denote also by $(e_j)_{j \in J(u)}$ the edges of $\Gamma_u$ 
   containing $u$. Each subtree $\Gamma_{u, e_j}$ is considered to be rooted 
   at the vertex of $e_j$ different from $u$. 
   Define recursively the {\bf continued fraction} $cf(\Gamma_u)$ of the 
   weighted rooted tree $\Gamma_u$  by: 
      \[
       cf(\Gamma_u)  = 
       \left\{ 
      \begin{array}{lcl}
               \alpha_u, &   \mbox{ if }  & 
                     \Gamma_u \mbox{ is reduced to the vertex } $u$,      \\
        \alpha_u -  {\sum_{j \in J(u)}  ( {cf(\Gamma_{u,e_j}) } )^{-1} }
                          &  & \mbox{ otherwise}. 
                   \end{array} 
                   \right.  
                  \]
    Then: 
        \[\det(\Gamma) = cf(\Gamma_u) \cdot \prod_{j \in J(u)} \det(\Gamma_{u,e_j}).\]           
\end{proposition}

The following multiplicative property of determinant products will be fundamental 
for us in the sequel:

\begin{proposition}  \label{lem:equalfund}
    For any three vertices $u,v,w \in \cV$ such that $v \in [uw]$, one has:
       \begin{equation} \label{eq:prod1} 
             p(u,v) \cdot p(v,w) = p(v,v) \cdot p(u,w).
       \end{equation}
     Equivalently:
        \begin{equation} \label{eq:prod2} 
              \frac{p(u,v)}{\sqrt{p(u,u) \cdot p(v,v)}} \cdot \frac{p(v,w)}{\sqrt{p(v,v) \cdot p(w,w)}} = 
                  \frac{p(u,w)}{\sqrt{p(u,u) \cdot p(w,w)}}. 
       \end{equation}
\end{proposition}

\begin{proof}  The following proof is to be followed on Figure \ref{fig:sk}. We define: 
     \begin{itemize}
        \item  $P(u) = $  the product of edge determinants 
            at the vertex $u$, over the set of edges starting from $u$ and 
            not contained in $[uw]$. The products $P(v)$ and $P(w)$ are defined analogously. 
        \item $P(uv) = $  the product of edge determinants at all vertices 
           of $\Gamma$ situated in the interior of the geodesic $[uv]$, over 
           the edges not contained in $[uw]$. $P(vw)$ is defined analogously. 
        \item $M = $  the edge determinant at $v$ in the direction of the unique edge 
           starting from $v$ and contained in $[uv]$. $N$ is defined analogously. 
     \end{itemize}

Then one has the following formulae, clearly understandable on Figure \ref{fig:sk}:
  $$ \begin{array}{lclcll}
          p(u,v) & = & P(u) \cdot P(uv) \cdot P(v) \cdot N, & \quad
          p(v,w) & = &  M \cdot P(v) \cdot P(vw) \cdot P(w), \\
          p(v,v) & = &  M \cdot N \cdot P(v), & \quad 
          p(u,w) &  = &  P(u) \cdot P(uv) \cdot P(v) \cdot P(vw) \cdot P(w).
   \end{array} $$
 The equality (\ref{eq:prod1}) is a direct consequence of the previous factorisations. 
 Finally, it is immediate to see that  (\ref{eq:prod1}) and (\ref{eq:prod2}) are equivalent. 
\end{proof}

\begin{figure}
\centering
\begin{tikzpicture} [scale=0.7]

    \draw[thick](0,0) -- (1.5,0) ;  
     \draw[thick,dotted] (1.5, 0) -- (2.5,0); 
       \draw[thick](2.5,0) -- (5.5,0) ;  
     \draw[thick,dotted] (5.5, 0) -- (6.5,0); 
    \draw[thick](6.5,0) -- (8, 0) ; 
     
    
     \draw(0,0) -- (0.25,1) ;  
     \draw[dotted] (-0.1, 0.75) -- (0.1,0.75);
     \draw(0,0) -- (-0.25,1) ;

     \draw(4,0) -- (4.25,1) ;  
     \draw[dotted] (3.9, 0.75) -- (4.1,0.75);
     \draw(4,0) -- (3.75,1) ;

     \draw(8,0) -- (8.25,1) ;  
     \draw[dotted] (7.9, 0.75) -- (8.1,0.75);
     \draw(8,0) -- (7.75,1) ;   
     
     
     \draw(1,0) -- (0.75,-1) ;  
     \draw[dotted] (0.9, -0.75) -- (1.1,-0.75);
     \draw(1,0) -- (1.25,-1) ;

     \draw(3,0) -- (3.25,-1) ;  
     \draw[dotted] (2.9, -0.75) -- (3.1,-0.75);
     \draw(3,0) -- (2.75,-1) ;

     \draw(5,0) -- (5.25,-1) ;  
     \draw[dotted] (4.9, -0.75) -- (5.1,-0.75);
     \draw(5,0) -- (4.75,-1) ;

     \draw(7,0) -- (7.25,-1) ;  
     \draw[dotted] (6.9, -0.75) -- (7.1,-0.75);
     \draw(7,0) -- (6.75,-1) ;

    \draw (-0.2,0) node[left]{$u$} ;
  \draw (8.2,0) node[right]{$w$} ;
 \draw (4,-0.2) node[below]{$v$} ;
 \draw (3.9,0.3) node[left]{$M$} ;
 \draw (4.1,0.3) node[right]{$N$} ;

  \draw (2, -1.45) node[below]{$P(uv)$} ;
  \draw (6, -1.45) node[below]{$P(vw)$} ;
  \draw (0, 1.45) node[above]{$P(u)$} ;
   \draw (4, 1.45) node[above]{$P(v)$} ;
    \draw (8, 1.45) node[above]{$P(w)$} ;
    
\draw (2, -1.15) node[below]{
    $\underbrace{ \, \,  \, \, \, \,  \, \,  \, \, \, \,  \, \,  \, \, \, \,  \, \,  \, \, \, \,   \, \,  \, \, \, \, }$} ;
\draw (6, -1.15) node[below]{
    $\underbrace{ \, \,  \, \, \, \,  \, \,  \, \, \, \, \, \,  \, \, \, \,  \, \,  \, \, \, \,   \, \,  \, \, \, \, }$} ;

\draw (0, 1.05) node[above]{$\overbrace{  \, \,  \, \, \, \,   \, \,   }$} ;
\draw (4, 1.05) node[above]{$\overbrace{  \, \,  \, \, \, \,   \, \,   }$} ;
\draw (8, 1.05) node[above]{$\overbrace{  \, \,  \, \, \, \,   \, \,   }$} ;
      \node[draw,circle,inner sep=1pt,fill=black] at (1,0) {};
       \node[draw,circle,inner sep=1pt,fill=black] at (3,0) {};
       \node[draw,circle,inner sep=1pt,fill=black] at (5,0) {};
        \node[draw,circle,inner sep=1pt,fill=black] at (7,0) {};

    \node[draw,circle,inner sep=2pt,fill=black] at (0,0) {};
     \node[draw,circle,inner sep=2pt,fill=black] at (4,0) {};
      \node[draw,circle,inner sep=2pt,fill=black] at (8,0) {};
    \end{tikzpicture}  
    \caption{Illustration for the proof of Proposition \ref{lem:equalfund}.}  \label{fig:sk}
    \end{figure}
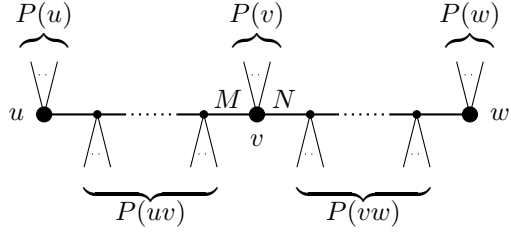

The following proposition was proved by Eisenbud and Neumann \cite[Lemma 20.2]{EN 85} 
(see also Neumann and Wahl \cite[Theorem 12.2]{NW 05}) for trees corresponding to the 
singularities whose boundaries $\partial M$ are rational homology spheres. 
Nevertheless, their proofs use only the fact that those are weighted trees 
appearing as dual graphs of singularities. 

\begin{proposition}  \label{prop:equlfund}
    Let $S$ be an arborescent singularity. Consider any good resolution of it. 
   Then, for any $v,  w \in \cV$, one has: 
              \[ {p(v, w)} = (- \check{E}_v \cdot \check{E}_w) \cdot {\det(S)}.\]
\end{proposition}

\begin{remark}  \label{rem:jonem}
    Using Proposition \ref{prop:equlfund}, formula (\ref{eq:prod1}) becomes the equality:  
    \begin{equation}  \label{eq:caseeq}
         (- \check{E}_u \cdot \check{E}_v) (- \check{E}_v \cdot \check{E}_w) = 
          (- \check{E}_v \cdot \check{E}_v) (- \check{E}_u \cdot \check{E}_w), 
          \mbox{ for any } v \in [uw]
    \end{equation}
    of the introduction. After having seen a previous version of this paper, Jonsson told us 
    that he had proved this equality for dual trees of compactifying divisors of $\C^2$ 
    and N\'emethi told us that the equality could also be proved using Lemma 4.0.1 
    from his paper \cite{BN 10} written with Braun.  
\end{remark}

\begin{remark}  \label{rem:pyth}
    Let us consider the real vector space $\Lambda_{\R}=\check{\Lambda}_{\R}$ 
    endowed with the 
    opposite of the intersection form. It is a Euclidean vector space. 
    Measuring the angles with this Euclidean metric, the equality (\ref{eq:caseeq}) becomes: 
       $$ \cos (\angle \check{E}_u \check{E}_v) \cdot \cos (\angle \check{E}_v \check{E}_w) 
                      =  
                \cos (\angle \check{E}_u \check{E}_w), \mbox{ for any } v \in [uw].$$
     By the spherical pythagorean theorem, this means that the spherical 
     triangle whose vertices are the unit vectors determined by $ \check{E}_u, 
      \check{E}_v,  \check{E}_w$ is rectangle at the vertex corresponding to $ \check{E}_v$. 
      This interpretation was noticed by the third author in the announcement 
      \cite{P 16} of the main results of this paper. 
\end{remark}

\begin{corollary} \label{cor:ineqd}
   For any $u,v \in \cV$, one has the inequality:
            $$  p(u,u) \cdot p(v,v) \geq p^2(u,v), $$
   with equality if and only if $u=v$. 
\end{corollary}

 \begin{proof}
        Using Proposition \ref{prop:equlfund}, the desired inequality may 
        be rewritten as:
            $$(-\check{E}_u \cdot \check{E}_u) \cdot (-\check{E}_v\cdot \check{E}_v) 
                \geq (-\check{E}_u\cdot \check{E}_v)^2$$
         which is simply the Cauchy-Schwartz inequality implied by the positive 
         definiteness of the intersection form $- I$ on $\check{\Lambda}_{\R}$. One has equality 
         if and only if the vectors $\check{E}_u$ and $\check{E}_v$ are proportional, which is 
         the case if and only if $u=v$. 
 \end{proof}
 
 Let us introduce a new function, which is well-defined thanks to Corollary \ref{cor:ineqd}: 
 
 \begin{definition} \label{def:distprod}
    The {\bf determinant distance} $d: \cV \times \cV \to \R_+$ is defined by:
       $$d(u,v) := - \log \frac{p(u,v)}{\sqrt{p(u,u) \cdot p(v,v)}}, \mbox{ for all }  u,v \in \cV.$$
 \end{definition}

 \begin{remark}  \label{rem:gendet}
     In the paper \cite{GR 17}, Gignac and Ruggiero defined the 
     \emph{angular distance} for any normal surface singularity, as:
         $$\rho(u,v) := - \log \frac{(-\check{E}_u \cdot \check{E}_v)^2}
         {(-\check{E}_u \cdot \check{E}_u) \cdot (-\check{E}_v \cdot \check{E}_v)}, 
         \mbox{ for all }  u,v \in \cV.$$
         By Proposition \ref{prop:equlfund}, on arborescent singularities one gets 
         $\rho =  2d$. In the sequel \cite{GGPR 17} to the present paper, concerning 
         arbitrary normal surface singularities, we also work with the angular distance 
         $\rho$ as a replacement of what we call here the \emph{determinant distance}. 
 \end{remark}

 Reformulated in terms of the determinant distance, equation (\ref{eq:prod2}) 
 provides:

 \begin{proposition}  \label{prop:equalcos}  
    For any three vertices $u,v,w \in \cV$ such that $v \in [uw]$, one has:
      $$d(u,v) + d(v,w) = d(u,w).$$
     Moreover $d$ is symmetric and $d(u,v) \geq 0$, with equality if and only 
     if $u=v$. That is, the determinant distance $d$ is an additive distance on the tree $\Gamma$, 
     in the sense of Definition \ref{def:addmet}.
\end{proposition}

Therefore, Proposition \ref{prop:chardist} implies that:

\begin{corollary}  \label{prop:ineqmin}
    Let $S$ be an arborescent singularity. Consider any good resolution of it. 
   Then, for any vertices $u,v,w,l \in \cV$, one has:
      \begin{equation}  \label{eq:addineq}
            d(l,u) + d(v,w) \leq 
            \max \{ d(l,v) + d(u,w), d(l,w) + d(u,v)  \}. 
       \end{equation}
    Equivalently:
         \begin{equation}  \label{eq:multineq}
            p(l,u) \cdot p(v,w) \geq 
            \min \{ p(l,v) \cdot p(u,w), p(l,w) \cdot p(u,v)  \}. 
       \end{equation}       
\end{corollary}

\begin{remark} \label{rem:cosine}
We could have worked instead with the function $e^{-d}$, which is a \emph{multiplicative 
distance function}, that is, a distance with values in the cancellative abelian 
monoid $( (0, + \infty), \cdot)$, in the sense of Bandelt and Steel \cite{BS 95}. We prefer 
to work instead with a classical additive distance, in order not to complicate the understanding  
of the reader who is not accustomed with this more general setting, which is generalized 
even more by B\"ocker and Dress \cite{BD 98}. Note also that, 
 as a consequence of Proposition \ref{prop:equlfund}, 
 one has:
      $$ e^{- d(u,v)} = 
          \frac{-(\check{E}_u \cdot \check{E}_v)}{\sqrt{ ( -(\check{E}_u \cdot \check{E}_u)) 
              \cdot (-(\check{E}_v\cdot \check{E}_v))} }, $$
 which is the cosine of the angle formed by the vectors $\check{E}_u$ 
 and $\check{E}_v$ with respect to the euclidean scalar product $-I$ on $\check{\Lambda}_{\R}$. 
 \end{remark}

Note the following consequence of Proposition \ref{prop:moreprec}, which refines  
inequality (\ref{eq:multineq}) from Corollary \ref{prop:ineqmin}:

\begin{proposition} \label{prop:refine} 
    For any $u,v,w,l \in \cV$, one has the equivalence:
      $$ p(l,v) \cdot p(u,w) \leq p(l,u) \cdot p(v,w)  \:  \Longleftrightarrow \:   [lv] \cap [uw] 
         \neq \emptyset$$
    with equality if and only if one has moreover $[lu] \cap [vw] \neq \emptyset$. 
\end{proposition}

This proposition or, alternatively, the weaker statement of Corollary \ref{prop:ineqmin}  
will imply in turn our main results   announced in the introduction (that is, 
Corollary \ref{cor:upbound}, Theorem \ref{thm:ultram} and Theorem \ref{thm:ordsemival}).

\begin{example} \label{ex:4p} 
   Let us consider a germ of arborescent singularity $S$ which 
   has a good resolution whose dual graph is indicated in Figure \ref{fig:edge-det}. 
   All self-intersections are equal to $-2$, with the exception of $E_f\cdot E_f = -3$. 
   The genera are arbitrary.  
The edge determinants at each vertex are indicated in Figure \ref{fig:edge-det} near the corresponding edge. For instance, $\det_{a, [ab]}(\Gamma)$ is the determinant 
of the subtree $\Gamma_{a, [ab]}$, which is the full subtree with vertices $e,b,f$. 
   This allows to compute the determinant product of any pair of vertices. For instance:   
        \begin{center}
           $  p(a,b) = \det_{a, [ac]}(\Gamma) \cdot  \det_{a, [ad]}(\Gamma)  \cdot
              \det_{b, [be]}(\Gamma) \cdot  \det_{b, [bf]}(\Gamma) = 2 \cdot 2 \cdot 2 \cdot 3 = 24. $
        \end{center}
    The matrix of determinant products $(p(u,v))_{u,v}$ is the following one, for the 
    ordering $a < \cdots < f$ of the vertices of $\Gamma$:  
\[
\begin{pmatrix}
28 & 24 &  14 & 14 & 12  & 8\\
24 & 24 & 12 & 12 &  12  &  8\\
14 & 12  &  9  & 7 &  6  & 4\\
14 &  12 &  7 &  9 &   6  & 4\\
12 &  12  & 6 &   6 &  8  & 4\\
8 &  8&   4  & 4 &  4 &   4
\end{pmatrix}. 
\]
Moreover, we have $\det(S)=4$, as may be computed easily using Proposition \ref{prop:duchon}. 
\end{example}

\begin{figure}
\centering
\begin{tikzpicture}

    \draw (0,0) -- (-1,1) ; 
    \draw (0,0) -- (-1,-1) ; 
    \draw (0,0) -- (2,0) ;
    \draw (2,0) -- (3,1) ; 
    \draw (2,0) -- (3,-1) ;
 \draw (-1,1) node[left]{$c$} ; 
     \draw (-1,-1) node[left]{$d$} ; 
     \draw (3,1) node[right]{$e$} ;
     \draw (3,-1) node[right]{$f$} ; 
     \draw (-0.1,0) node[left]{$a$} ;
     \draw (2.2,0) node[right]{$b$} ; 
     
\draw (-0.6,0.7) node[above]{$9$} ;
\draw (-0.1,0.1) node[above]{$2$} ;   
\draw (-0.1,-0.1) node[below]{$2$} ; 
\draw (-0.9,-0.9) node[right]{$9$} ; 
\draw (0.1,0.2) node[right]{$7$} ; 
\draw (1.9,0.2) node[left]{$4$} ;       
\draw (2.1,0.1) node[above]{$2$} ;   
\draw (2.1,-0.1) node[below]{$3$} ;   
\draw (2.9,0.9) node[left]{$8$} ;  
\draw (2.9,-0.9) node[left]{$4$} ;  

\node[draw,circle,inner sep=1pt,fill=black] at (0,0) {};

\node[draw,circle,inner sep=1pt,fill=black] at (-1,1) {};

\node[draw,circle,inner sep=1pt,fill=black] at (3,1) {};

\node[draw,circle,inner sep=1pt,fill=black] at (3,-1) {};

\node[draw,circle,inner sep=1pt,fill=black] at (-1,-1) {};

\node[draw,circle,inner sep=1pt,fill=black] at (2,0) {};

    \end{tikzpicture}  
    \caption{The edge determinants in Example \ref{ex:4p}. } 
    \label{fig:edge-det}
    \end{figure}
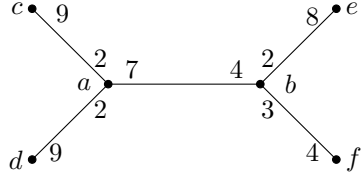
    
       \begin{remark}  \label{compnextart}
            Specialize Proposition \ref{prop:refine} by putting  $v=w$.  
             Then one has automatically  
           $[lv] \cap [uw]  \neq \emptyset$, which implies:
               $$ p(l,v) \cdot p(u,v) \leq p(l,u) \cdot p(v,v)$$
         for any $l, u, v \in \cV$. This inequality may be rewritten as:
         \begin{equation} \label{eq: GR}
               (- \check{E}_l \cdot \check{E}_v) (- \check{E}_u \cdot \check{E}_v) \leq 
          (- \check{E}_l \cdot \check{E}_u) (- \check{E}_v \cdot \check{E}_v),
          \end{equation}
        with equality if and only if $v \in [lu]$. 
       In \cite{GR 17}, Gignac and Ruggiero proved that 
        the inequality (\ref{eq: GR}) is valid for \emph{all} normal surface singularities, 
        and moreover that the equality in (\ref{eq: GR}) holds if and only if 
        \emph{$v$ separates $u$ from $w$ in the dual graph}. 
        This condition is a generalization of the condition $v \in [lu]$ when $S$ is arborescent. 
        Their result is the central ingredient of the sequel \cite{GGPR 17} of the present paper, 
        written in collaboration with Ruggiero.  
    \end{remark}

\medskip
\subsection{The ultrametric associated to a branch on an arborescent singularity}
$\:$ 
\medskip

The main results of this subsection are our generalization of P\l oski's 
theorem to arbitrary arborescent singularities (Theorem \ref{thm:ultram}) and the interpretation 
of the associated rooted trees as convex hulls in dual graphs of resolutions 
(Theorem \ref{thm:topint}).   

\medskip

Recall the notation $U_L$ explained in the introduction: 

\begin{notation} \label{not:ultram}
   Let $S$ be a normal surface singularity and let $L$ be a 
  fixed branch on it. If $A, B$ denote two branches on $S$ different from $L$, define: 
    $$ U_L(A, B):= \left\{ \begin{array}{lcl}
              \dfrac{(L \cdot A) \: (L \cdot B)}{A \cdot B}, &  \mbox{ if } &  A \neq B \\
              0,  & \mbox{ if } &  A = B . 
                    \end{array} \right. $$
\end{notation}

The following proposition explains several ways to compute or to think about $U_L$ 
in the case of \emph{arborescent} singularities 
(recall that the notation $u(A)$ was introduced in Definition \ref{def:embres}): 

\begin{proposition}  \label{prop:ultramform}
   Let $S$ be an \emph{arborescent} singularity and let $L$ be a 
  fixed branch on it. Assume moreover that  $A, B$ are two distinct branches on $S$ 
  different from $L$ and that $\tilde{S}$ is an embedded resolution of $A + B +L$, with 
  dual tree $\Gamma$. 
   Denote $a = u(A), b= u(B), l= u(L)$. Then:
    \begin{enumerate}
         \item[1.] \label{it1} 
              $ U_L(A,B)= \det(S)^{-1}  \cdot \dfrac{p(l, a) \cdot p(l, b)}{p(a, b)}.$
         \item[2.] \label{it2} 
            $ U_L(A,B)= 
                 \det(S)^{-1}  \cdot \dfrac{p^2(l, a \wedge_l b)}{p(a \wedge_l b , a \wedge_l b)}.$ 
         \item[3.] \label{it3} 
             $ U_L(A,B)=    \det(S)^{-1}  \cdot {p(l,l)} \cdot e^{- 2 h_l(a \wedge_l b)}$, 
                where $h_l$ is the remoteness function on $\Gamma$ associated to the 
                  determinant distance $d$ and the root $l$, as explained in Definition \ref{def:heightbis}.
    \end{enumerate}
\end{proposition}

\begin{proof} 
{\bf Let us prove point 1}.  By Corollary \ref{cor:invint}, $A \cdot B = 
        - \check{E}_a \cdot \check{E}_b$. 
        By Proposition \ref{prop:equlfund}, 
        $- \check{E}_a \cdot \check{E}_b = \det(S)^{-1} \cdot p(a,b)$. 
        Using the analogous formulae in order to transform the intersection numbers 
        $L \cdot A$ and $L \cdot B$, we get the desired equality. 
    
\medskip
{\bf We prove now  point 2}.
    Given the equality of the previous point, the second 
      equality is equivalent to: 
         $$ p(l, a) \cdot p(l, b) \cdot p(a \wedge_l b , a \wedge_l b) = 
              p(a,b) \cdot p^2(l, a \wedge_l b).$$
       But this may be obtained by multiplying termwise the following special cases of 
       formula (\ref{eq:prod1}) stated in Proposition \ref{lem:equalfund} 
       (in which, for simplicity, we have denoted $c:= a \wedge_l b$):
         \[ \begin{array}{lcl}
                    p(l,a) \cdot p(c,c) & = &  p(l,c) \cdot p(c,a),\\
                    p(l,b) \cdot p(c,c)  & = & p(l,c) \cdot p(c,b),\\
                    p(a,c) \cdot p(c,b) & = & p(a,b) \cdot p(c,c).
              \end{array} \]
    
    \medskip
   {\bf Finally, let us prove point 3}. Using Definition \ref{def:distprod}, 
      the equality (\ref{it2}) may be rewritten as:
         $$U_L(A,B)= \dfrac{p(l,l)}{ \det(S)} 
                   \cdot \dfrac{p^2(l, a \wedge_l b)}{p(l,l) \cdot p(a \wedge_l b , a \wedge_l b)}  = 
                      \dfrac{p(l,l)}{ \det(S)} 
                   \cdot e^{- 2 d(l,  a \wedge_l b)}. $$
        But, by Definition \ref{def:heightbis}, $d(l, a \wedge_l b) = h_l(a \wedge_l b)$, 
        and the formula is proved. 
\end{proof}

The first equality stated in Proposition \ref{prop:ultramform} allows to compute the 
maximum of $U_L$:

\begin{corollary}  \label{cor:upbound}
    Whenever $L, A, B$ are three pairwise distinct branches on the arborescent  
    singularity $S$, one has: 
        $$U_L(A,B) \leq - \check{E}_l \cdot \check{E}_l,$$
    with equality if and only if $l \in [ab]$. 
\end{corollary}

  \begin{proof}
      As $[la] \cap [lb] \neq \emptyset$, Proposition \ref{prop:refine} implies that:
        $p(l,a) \cdot p(l,b) \leq p(l,l) \cdot p(a,b). $
      Combining this with the first equality stated in Proposition \ref{prop:ultramform}, we get:
         $$U_L(A,B) \leq {p(l,l)} \cdot { \det(S)}^{-1} = - \check{E}_l \cdot \check{E}_l,$$
       where the last equality is a consequence of Proposition \ref{prop:equlfund}.    
       The fact that one has equality if an only if $l \in [ab]$ follows from Proposition 
       \ref{prop:refine}. 
  \end{proof}

Recall that $\cB(S)$ denotes the set of branches on $S$. The following 
is our generalization of P\l oski's theorem recalled at the beginning of the introduction: 

\begin{theorem}  \label{thm:ultram}
   For any four pairwise distinct branches $(L, A, B, C)$ on the arborescent singularity 
   $S$, one has:
      $$U_L(A,B)  \leq \max \{ U_L(A, C), U_L(B, C) \}.$$
      Therefore, the function $U_L$ is 
      an ultrametric on $\cB(S) \setminus \{ L \}$. 
\end{theorem}

\begin{proof} We will give two different proofs of this theorem. 
   
    \medskip
    \noindent
     {\bf The first proof.} 
      Let $\pi: \tilde{S} \to S$ be an embedded resolution 
      of $A+B+C+L$. By Proposition \ref{prop:intexcep}, we know that 
      the pairwise intersection numbers on $S$ of the four branches are the 
      opposites of the intersection numbers on $\tilde{S}$ of their exceptional 
      transforms by the morphism $\pi$. 
      By Lemma \ref{lem:excdual}, there exist four (possibly coinciding) indices 
      $l,a,b,c \in \cV$ such that $(\pi^*A)_{ex} = - \check{E}_a, (\pi^*B)_{ex} = -\check{E}_b, 
          (\pi^*C)_{ex} = -\check{E}_c,  (\pi^*L)_{ex} = -\check{E}_l$. 
      Using the first equality of Proposition \ref{prop:ultramform}, 
      the stated inequality is equivalent to:
         $$ \frac{p(l, a) \cdot p(l, b)}{ p(a, b)}  
           \leq  \max \left\{\frac{p(l, a) \cdot p(l, c)}{ p(a, c)}  ,  
                  \frac{p(l, b) \cdot p(l, c)}{ p(b, c)}\right\}. $$
        By taking the inverses of the three fractions and multiplying then all of them by 
        $p(l, a) \cdot  p(l, b) \cdot p(l, c)$, we see that the previous 
        inequality is equivalent to:
          $$ p(l,c) \cdot p(a,b) \geq \min \{ p(l,b) \cdot p(a,c),  \:  p(l,a) \cdot p(b,c)\}.$$
        But this inequality is true by Corollary \ref{prop:ineqmin}. 
        
        \medskip
        \noindent
        {\bf The second proof.} 
       We could have argued also by combining the third equality of 
        Proposition \ref{prop:ultramform} with Proposition \ref{prop:addistultram}. 
        Indeed, the function $s: (\R_+, \leq) \to (\R_+^*, \leq)$ defined by:
            $s(x) :=  \frac{p(l,l)}{ \det(S)} \cdot e^{-2 x}$,
        is strictly decreasing. 
        This line of reasoning may be easily followed on Figure \ref{fig:abcl}. 
        Up to permuting $a, b, c$,  it represents the generic tree $[labc]$. All other 
        topological possibilities are degenerations of it. Using the third equality of 
        Proposition \ref{prop:ultramform}, we have: 
            \[\begin{array}{rcl}
                       U_L(A,B) & =  & { \det(S)}^{-1} \cdot {p(l,l)} \cdot e^{- 2 h_l(a \wedge_l b)}, \\
                       U_L(A,C)= U_L(B,C) & = & 
                            { \det(S)}^{-1} \cdot {p(l,l)} \cdot e^{- 2 h_l(b \wedge_l c)}. 
                 \end{array} \]
        But  the inequality $b \wedge_l c \preceq_l a \wedge_l b$ implies that:
           $ h_l(b \wedge_l c) \leq h_l(a \wedge_l b). $
        Therefore: 
               $$U_L(A,C)= U_L(B,C) \geq U_L(A,B),$$
         which is the ultrametric inequality (recall Proposition \ref{prop:equivultra} (\ref{second})).
    \end{proof}
               
        \begin{figure}
\centering
\begin{tikzpicture} [scale=0.8]
 \draw(0,0) -- (0,-2) ;  
 \draw(0,0) -- (1,1);
 \draw(0,0) -- (-1,1) ;
  \draw(0,-1) -- (2,1) ; 
   
 \draw (-0.1, 0) node[left]{$a \wedge_l b$} ;
 \draw (-0.1, -1) node[left]{$a \wedge_l c = b \wedge_l c$} ;
 \draw (1, 1) node[above]{$b$} ;
\draw (-1,1) node[above]{$a$} ;
\draw (2,1) node[above]{$c$} ;
\draw (0,-2) node[below]{$l$} ;

 \node[draw,circle,inner sep=1pt,fill=black] at (0,0) {};
 \node[draw,circle,inner sep=1pt,fill=black] at (1,1) {};
\node[draw,circle,inner sep=1pt,fill=black] at (0,-2) {};
\node[draw,circle,inner sep=1pt,fill=black] at (-1,1) {};
\node[draw,circle,inner sep=1pt,fill=black] at (2,1) {};
\node[draw,circle,inner sep=1pt,fill=black] at (0,-1) {};

    \end{tikzpicture}  
    \caption{A generic position of $a, b, c$ and $l$.}  
    \label{fig:abcl}
    \end{figure}
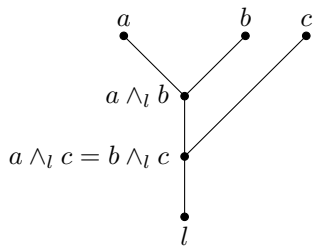

\begin{remark}  \label{rem:histultra} 
    The previous theorem was proved in this form for \emph{smooth} 
complex germs $S$ and $L$ by Ch\c adzy\'nski and P\l oski 
\cite[Section 4]{CP 88} and again  by Favre and Jonsson in 
\cite[Lemma 3.56]{FJ 04} for \emph{smooth} germs $S$ and $L$.  
Ab\'{\i}o, Alberich-Carrami\~{n}ana and 
Gonz\'alez-Alonso later explored in \cite{AAG 11} the values taken by this 
ultrametric, also in the case of smooth germs.  Their results were extended 
recently by the first author and P\l oski \cite{GP 17} 
to smooth surfaces defined over algebraically closed fields of positive characteristic. 
Favre and Jonsson were not conscious about 
the result of Ch\c adzy\'nski and P\l oski, and they attributed the theorem to 
the first author's thesis \cite[Cor. 1.2.3]{G 96}. This last reference  
states in fact that a related function is an ultrametric, a result which 
combined with \cite[Prop. 1.2.2]{G 96} implies indeed that $U_L$ is an ultrametric. 
Note that at the time of writing  \cite{G 96}, the first author did not know the papers 
\cite{P 85} and \cite{CP 88}.  
\end{remark}

As a consequence of Proposition \ref{prop:embedtree}, one gets also 
the announced topological interpretation of the two rooted trees associated 
to the ultrametric $U_L$ (see Definition \ref{def:asstree}):

\begin{theorem}  \label{thm:topint} 
   Let $S$ be an arborescent singularity and $L$ a fixed branch on it. Assume that 
 $\cF$ is a \emph{finite} set of branches on $S$, all of them 
 different from $L$. Consider an embedded resolution of the sum of $L$ and of the branches 
     in $\cF$. Let  $\Gamma_L(\cF)$ be the dual tree of the total transform 
     of this divisor. Then we have:
         \begin{enumerate}
             \item  the end-rooted tree $\hat{T}^{U_L}$ associated to the ultrametric space 
     $(\cF, U^L)$ is isomorphic as a rooted tree with leaf set $\cF$  
     with the convex hull of $\{L\} \cup \cF$  
     in $\Gamma_L(\cF)$, endowed with its topological vertex set and with root at $L$; 
             \item  the previous isomorphism sends the interior-rooted tree $T^{U_L}$ 
               associated to the ultrametric space $(\cF, U^L)$ onto the convex hull 
               of $\cF$ in $\Gamma_L(\cF)$. 
        \end{enumerate}
\end{theorem}

\begin{remark} \label{rem:rootend}
   Note that the root $L$ and the branches in $\cF$ are always ends of $\Gamma_L(\cF)$. 
   This is the reason why we have decided to associate systematically  to 
   an ultrametric a rooted tree in which the root is an end, its \emph{end-rooted tree} 
    (see Definition \ref{def:asstree}). Note also that the convex hull 
    $[\{u(L)\}  \cup \{u(A), \:   A \in \cF\}]$ inside $\Gamma_L(\cF)$, which is equal to the core 
    of the end-rooted tree $\Gamma_L(\cF)$, is equipped with the additive distance $d$ 
    of Definition \ref{def:distprod}. This fact has to be used 
    when one deduces Theorem \ref{thm:topint} from Proposition \ref{prop:embedtree}. 
\end{remark}

\begin{example}   \label{ex:twoval}
   Let us consider an arborescent singularity $S$ as in Example \ref{ex:4p}. That is, 
   we assume that it admits a good resolution $\tilde{S}$  with weighted dual graph $\Gamma$ 
   as in Figure \ref{fig:edge-det}. 
   Consider branches $L, A, B, C, D, E, F$ on $S$ such that the total transform of 
   $L + A + \cdots + F$ on $\tilde{S}$ is a normal crossings divisor. Moreover, we assume that 
   the strict transforms of $A, ..., F$ intersect $a, ..., f$ respectively, and that the 
   strict transform of $L$ intersects $E_a$. Therefore, with the notations of Proposition 
   \ref{prop:ultramform}, we have $l=a$. We see on Figure \ref{fig:edge-det} that when 
   $x,y$ vary among $\{a, ..., f\}$ and remain distinct, their infimum $x \wedge_l y$ 
   relative to the root $l=a$ of $\Gamma$ is either $a$ or $b$. By the second 
   equality of  Proposition  \ref{prop:ultramform}, we deduce that the only 
   possible values of $\det(S) \cdot U_L(X,Y)$, when $X \neq Y$ vary among 
   $\{A, ..., F\}$, are: 
   \[          \dfrac{p^2(a,a)}{p(a,a)} = p(a,a) = 28, \quad
                      \dfrac{p^2(a,b)}{p(b,b)} = \dfrac{24^2}{24} = 24.
    \]
   Continuing to use the second equality of  Proposition  \ref{prop:ultramform}, we get 
   the following values for the entries of the matrix $(U_L(X,Y))_{X,Y}$ 
   (recall from Example  \ref{ex:4p} that $\det(S)=4$): 
   \[
\begin{pmatrix}
0 & 7 &  7 & 7 & 7  & 7\\
7 & 0 & 7 & 7 &  6  &  6\\
7 & 7  &  0  & 7 &  7  & 7\\
7 &  7 &  7 &  0 &   7  & 7\\
7 &  6  & 7 &   7 &  0  & 6\\
7 &  6 &   7  & 7 &  6 &   0
\end{pmatrix}. 
\]
One may check immediately on this matrix that $U_L$ is an ultrametric 
on the set $\{A, ..., F\}$. 
In Figure \ref{fig:both} we represent 
   both the dual tree $\hat{\Gamma}$ of the total transform of $L + A + \cdots + F $ and 
   the end-rooted tree $\hat{T}^{U_L}$ associated to the ultrametric $U_L$. Near the two nodes 
   of $\hat{T}^{U_L}$ we  
   indicate both the corresponding clusters and their diameters (as in Figure 
   \ref{fig:assoctree}).  
\end{example}

   \begin{figure}
\centering
\begin{tikzpicture} [scale=0.8]

    \draw (0,0) -- (-1,1) ; 
    \draw (0,0) -- (-1,-1) ; 
    \draw (0,0) -- (2,0) ;
    \draw (2,0) -- (3,1) ; 
    \draw (2,0) -- (3,-1) ;
     \draw (-1.2 ,1) node[left]{$c$} ; 
     \draw (-1.2 ,-1) node[left]{$d$} ; 
     \draw (3.2 ,1) node[right]{$e$} ;
     \draw (3.2 ,-1) node[right]{$f$} ; 
     \draw (-0.1,0) node[left]{$a$} ;
     \draw (2.2,0) node[right]{$b$} ;

      \draw [->] (0,0) -- (0,1) ; 
    \draw[very thick] [->]  (0,0) -- (0,-1) ; 
    \draw [->]  (-1,1) -- (-1,2) ;
    \draw [->]  (-1,-1) -- (-1,-2) ; 
    \draw [->]  (2,0) -- (2,1) ;
     \draw [->]  (3,1) -- (3,2) ; 
    \draw [->]  (3,-1) -- (3,-2) ;

      \draw (-1 ,2) node[above]{$C$} ; 
     \draw (-1 ,-2) node[below]{$D$} ; 
     \draw (3 ,2) node[above]{$E$} ;
     \draw (3 ,-2) node[below]{$F$} ; 
     \draw (0,1) node[above]{$A$} ;
     \draw (2,1) node[above]{$B$} ; 
      \draw (0,-1) node[below]{$L$} ; 
 
\node[draw,circle,inner sep=1pt,fill=black] at (0,0) {};

\node[draw,circle,inner sep=1pt,fill=black] at (-1,1) {};

\node[draw,circle,inner sep=1pt,fill=black] at (3,1) {};

\node[draw,circle,inner sep=1pt,fill=black] at (3,-1) {};

\node[draw,circle,inner sep=1pt,fill=black] at (-1,-1) {};

\node[draw,circle,inner sep=1pt,fill=black] at (2,0) {};


 \draw(9, -2) -- (9, -1) ;  
 \draw(9, -1) -- (7, 1);
 \draw(9, -1) -- (9, 2) ;
  \draw(9, - 1) -- (10, 2) ; 
  \draw(9, -1) -- (11, 2) ;  
 \draw(7, 1) -- (6,  2);
 \draw(7, 1) -- (7, 2) ;
  \draw(7,  1) -- (8,  2) ; 
   
 \draw (6,  2) node[above]{$\{B\}$} ;
 \draw (7, 2) node[above]{$\{E\}$} ;
 \draw (8, 2) node[above]{$\{F\}$} ;
\draw (9, 2) node[above]{$\{A\}$} ;
\draw (10, 2) node[above]{$\{C\}$} ;
\draw (11, 2) node[above]{$\{D\}$} ;
\draw (9, -2.2) node[below]{$\hat{\rho}$} ;

 \draw (7, 1) node[left]{$\{B, E, F \}$} ;
 \draw (9, -1) node[left]{$\{A, B, C, D, E, F \}$} ;
 
  \draw (7.2 , 1) node[right]{$6$} ;
 \draw (9.2 , -1) node[right]{$7$} ;

 \node[draw,circle,inner sep=1pt,fill=black] at (9 , -2) {};
 \node[draw,circle,inner sep=1pt,fill=black] at (9 , -1) {};
\node[draw,circle,inner sep=1pt,fill=black] at (7 , 1) {};
\node[draw,circle,inner sep=1pt,fill=black] at (6  , 2) {};
\node[draw,circle,inner sep=1pt,fill=black] at (7 , 2) {};
\node[draw,circle,inner sep=1pt,fill=black] at (8 , 2) {};
\node[draw,circle,inner sep=1pt,fill=black] at (9 , 2) {};
\node[draw,circle,inner sep=1pt,fill=black] at (10 , 2) {};
\node[draw,circle,inner sep=1pt,fill=black] at (11 , 2) {};

    \end{tikzpicture}  
    \caption{An illustration of Theorem \ref{thm:topint}: Example \ref{ex:twoval}}  
    \label{fig:both}
    \end{figure}
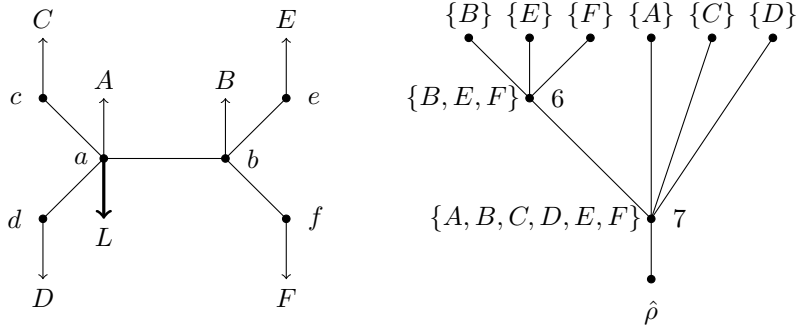

\medskip
In the introduction we recalled the following result of Teissier \cite[page 40]{T 77}, 
which inspired us to formulate Corollary \ref{cor:upbound}:

\begin{proposition}  \label{prop:teissier}
      If $S$ is a normal surface singularity, $A, B$ are two divisors 
  without common branches on it and $m_O$ denotes the function which 
  gives the multiplicity at $O$, then one has the inequality:
      $$\frac{m_O(A) \cdot m_O(B)}{A \cdot B} \leq m_O(S).$$
\end{proposition}

This result suggests to consider the following analog of the function $U_L$ 
introduced in Notation \ref{not:ultram}:

\begin{notation} \label{not:invultra}
        Let $S$ be a normal surface singularity. If $A, B$ denote two 
        branches on $S$, define: 
    $$ U_O(A, B):= \left\{ \begin{array}{lcl}
              \dfrac{m_O(A) \cdot m_O(B)}{A \cdot B}, &  \mbox{ if }  & A \neq B, \\
              0 , \:  & \mbox{ if } & A = B . 
                    \end{array} \right. $$
\end{notation}

An immediate consequence of the previous results is the following analog 
of Theorem \ref{thm:ultram} (which holds for a restricted class of arborescent 
singularities):

\begin{theorem}  \label{thm:ultraint}
   For any three pairwise distinct branches $(A, B, C)$ on the arborescent singularity 
   $S$ \emph{with irreducible generic hyperplane section}, one has:
      $$U_O(A,B)  \leq \max \{ U_O(A, C), U_O(B, C) \}.$$
      Therefore,  the function $U_O$  is an ultrametric on the set $\cB(S)$ of branches of $S$. 
\end{theorem}

\begin{proof}
   By definition, a \emph{generic hyperplane section} of a normal surface 
   singularity $S$ is the divisor defined by a generic element of the 
   maximal ideal of the local ring of $S$. For instance, if $S$ is 
   smooth, the generic hyperplane sections are  smooth branches 
   on $S$. 
   
    Fix an embedding of $S$ in a smooth space $(\C^n,0)$. Choose a generic hyperplane 
     $H$ in this space which is transversal to the three branches $A, B, C$. Therefore, 
     its intersection numbers with the branches are equal to their multiplicities. Denote 
     by $L$ the intersection of $S$ and $H$, which is a branch by hypothesis. 
     Since the intersection number of a branch with $H$ in the ambient smooth space 
     $(\C^n,0)$ is equal to the intersection number of the branch with $L$ on 
     $S$, we get:
       	\begin{equation} \label{mO}
	m_O(A) = L \cdot A, \quad  m_O(B) = L \cdot B, \quad  m_O(C) = L \cdot C. 
	\end{equation}
     Therefore:
       $$U_O(A, B) = U_L(A,B), \quad  U_O(A, C) = U_L(A,C), \quad U_O(B, C) = U_L(B,C).$$
     We conclude using Theorem \ref{thm:ultram}. 
\end{proof}

\begin{remark}  \label{rem:ratirr}
   Assume that $S$ is a \emph{rational} surface singularity and that 
   $\tilde{S}$ is a resolution of it.  Then the exceptional transform 
   on $\tilde{S}$ of a generic hyperplane section $L$ of $S$ is the \emph{fundamental 
   cycle} $Z_f$ of the resolution, defined by Artin \cite[Page 132]{A 66} (see also 
   Ishii \cite[Definition 7.2.3]{I 14}). The total transform of $L$ is in this case a 
   normal crossings divisor. The number of branches of $L$ whose strict 
   transforms intersect a component $E_u$ of $E$ is equal to the intersection 
   number $- Z_f \cdot E_u$. Therefore, the generic hyperplane section is irreducible  
   if and only if all these numbers vanish, with the exception of one 
   of them, which is equal to $1$ (that is, if and only if there exists $u \in \cV$ 
   such that $Z_f = - \check{E}_u$). This may be easily checked. For instance, 
   starting from the list of rational surface singularities of multiplicity $2$ or $3$ 
   given at the end of Artin's paper \cite{A 66}, one sees that among the rational 
   double points $A_n, D_n, E_n$, only those of type $A_n$ do not have irreducible 
   generic hyperplane sections. And among rational triple points, those which 
   do not have irreducible generic hyperplane sections are the first three of the left 
   column and the first one of the right column of that paper. 
 \end{remark}

\begin{remark} \label{rem:partcase}
    Under the hypothesis of Theorem \ref{thm:ultraint}, Teissier's inequality  
    stated in Proposition \ref{prop:teissier} may be proved as a particular case 
    of the inequality stated in Corollary \ref{cor:upbound}. Indeed, arguing as in the 
    proof of Theorem \ref{thm:ultraint}, we may assume that we work 
    with an irreducible hyperplane section $L$ such that the equalities (\ref{mO}) hold. 
    Let $L'$ be a second hyperplane section, whose strict transform to the 
    resolution with which we work is disjoint from the strict transform of $L$, 
    but intersects the same component $E_l$. Moreover, we may assume that 
    both are transversal to $E_l$. By Corollary \ref{cor:invint}, we have 
    $L \cdot L' = - E_l^2$. But $L \cdot L' = m_O(S)$. This shows, as announced, that 
    our inequality becomes  Teissier's one. 
\end{remark}

\medskip
\subsection{P\l oski's theorem and the ultrametric nature of Eggers-Wall trees}  
\label{aplloski}
$\: $ 
\medskip

In this subsection we assume that both the surface $S$ and the branch $L$ are smooth. 
Consider a finite set $\cF$ of branches on $S$, distinct from $L$. 
We explain first how to associate to them a rooted tree 
$\Theta_L(\cF)$, whose set of leaves is labeled by the elements 
of $\cF$ and whose root is labeled by $L$: the \emph{Eggers-Wall 
tree of $\cF$ relative to $L$}. Then we prove that the function $U_L$ 
is an ultrametric on $\cF$ with associated end-rooted tree 
isomorphic to $\Theta_L(\cF)$, by showing that in restriction 
to $\cF$, the function $U_L$ corresponds to a depth function on 
$\Theta_L(\cF)$. Note that the content of this subsection cannot be 
extended to algebraically closed fields of positive characteristic, because 
in this setting there are Weiertrass polynomials whose roots are not 
expressible as Newton-Puiseux series  (see the related Remark \ref{Campillo}). 
\medskip

 Choose a coordinate system $(x,y)$ on $S$  such that $L$ is defined by $x=0$. 

The following considerations on characteristic exponents are classical. 
One may find information about their historical evolution in \cite[Section 2]{GBGPPP 17}.  

 Let $A$ be a branch on $S$ different 
 from $L$. Relative to the coordinate system $(x,y)$, it 
 may be defined by a Weierstrass polynomial 
 $f_A \in \C [[x]][y]$, which is unitary, irreducible and of degree $d_A= L \cdot A$. 
 For simplicity, we mention only the dependency on $A$, not on the coordinate system $(x,y)$. 
 
 By the Newton-Puiseux theorem, $f_A$ has $d_A$ roots inside $\C [[x^{1/ d_A}]]$ 
 (the {\bf Newton-Puiseux roots} of $A$ in the coordinate system $(x,y)$), 
 which may be obtained from a fixed one of them $\xi(x^{1/ d_A})$ by 
 replacing $x^{1/ d_A}$ with the other $d_A$-th roots of $x$ (here $\xi \in \C [[t]]$ denotes 
 a formal power series with non-negative integral exponents). Therefore, all 
 the Newton-Puiseux roots have the same exponents. Some of those exponents 
 may be distinguished by looking at the differences of roots:

 \begin{definition}  \label{def:charexp}
     The {\bf characteristic exponents} of $A$ relative to $L$ are the $x$-orders 
     $\nu_x(\eta - \eta')$ 
     of the differences between Newton-Puiseux roots $\eta, \eta'$ 
     of $A$ in the coordinate system $(x,y)$. 
 \end{definition}
 
 The characteristic exponents may be read from a given Newton-Puiseux root 
 $\eta \in \C [[x^{1/ d_A}]]$ of $f_A$ by looking at the increasing sequence of exponents 
 appearing in $\eta$ and by keeping those which cannot be written as a quotient of 
 integers with the same smallest common denominator as the previous ones. In this 
 sequence, one starts from the first exponent which is not an integer. 
 
 The \emph{Eggers-Wall tree of $A$ relative to $L$} is a geometrical way of encoding 
 the sequences of characteristic exponents and of their successive 
 common denominators:

 \begin{definition} \label{def:EWirr}
     The {\bf Eggers-Wall tree} $\Theta_L(A)$ relative to $L$ 
     is a compact oriented segment endowed with the following 
     supplementary structures: 
         \begin{itemize} 
             \item  an increasing homeomorphism $\ex_{L,A} : \Theta_L(A) \to [0, \infty]$, 
                the {\bf exponent function}; 
             \item  marked points, which are by definition the points whose exponents 
                 are the characteristic exponents of $A$ relative to $L$, as well as the smallest 
                 end of $\Theta_L(A)$, labeled by $L$, and the greatest point, labeled by $A$. 
             \item an {\bf index function}  $\de_{L,A}: \Theta_L(A) \to \N$, which associates 
                 to each point $P \in \Theta_L(A)$ the index of $(\Z, +)$ in the 
                 subgroup of $(\Q, +)$ generated by the characteristic exponents 
                 of $A$ which are strictly smaller than $\ex_{L,A}(P)$. 
          \end{itemize}
 \end{definition}
 
 The index $\de_{L, A}(P)$ may be also seen as the smallest common denominator 
 of the exponents of a Newton-Puiseux root of $f_A$ which are strictly less than 
 $\ex_{L,A}(P)$. 
 
 \medskip

  \begin{example} \label{exchar}
       Assume that $A$ has as Newton-Puiseux root $x + x^{5/2} + x^{8/3} +  x^{17/6}$.  
       The set of characteristic exponents 
       of $A$ relative to the branch $L$ defined by $x=0$ is $\cE(A) = \{5/2, 8/3\}$. 
     The Eggers-Wall tree $\Theta_L(A)$ 
     is drawn in Figure \ref{fig:EWbranch}. We wrote the value of the exponent 
     function near each vertex and of the index function near each edge 
     on which it is constant.
  \end{example}

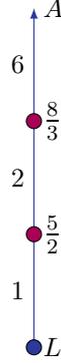
\begin{figure}
\centering
\begin{tikzpicture}[scale=1.5]
\draw [->, >=latex, color={rgb:red,45;green,58;blue,162}](0,0) -- (0,3);
\node[draw,circle,inner sep=2pt,fill={rgb:red,45;green,58;blue,162}] at (0,0){};
\node[draw,circle,inner sep=2pt,fill={rgb:red,255;green,0;blue,128}] at (0,1) {};
\node[draw,circle,inner sep=2pt,fill={rgb:red,255;green,0;blue,128}] at (0,2) {};
\draw (0,0) node[right]{$L$} ;
\draw (0,1) node[right]{\Large$\frac{5}{2}$} ;
\draw (0,2) node[right]{\Large$\frac{8}{3}$} ;
\draw (0,0.5) node[left]{$1$} ;
\draw (0,1.5) node[left]{$2$} ;
\draw (0,2.5) node[left]{$6$} ;
\draw (0,3) node[right]{$A$} ;
\end{tikzpicture}
   \caption{The Eggers-Wall tree of the series of Example \ref{exchar}} 
   \label{fig:EWbranch}
\end{figure}

We give now the definition of the Eggers-Wall tree associated 
to several branches. In addition to the characteristic exponents of the 
individual branches, we need to know the \emph{orders of coincidence} 
of the pairs of branches:

\begin{definition}  \label{def:ordcoin}
 If $A$ and $B$ are two distinct branches, which are also distinct from $L$, 
 then their {\bf order of coincidence} relative to $L$ is defined by:
      $$k_L(A, B):=\max  \{ \nu_x(\eta_A - \eta_B) \: | \: \eta_A  
         \in  Z(f_A),   \: \:    \eta_B  \in Z(f_B) \}  \in  \Q_+^* . $$
\end{definition}

Informally speaking, the order of coincidence is the greatest rational number $k$ for which 
one may find Newton-Puiseux roots of the two branches coinciding up to that 
number ($k$ excluded). 

Notice that the order of coincidence is symmetric: $k_L(A,B) = k_L(B,A)$.

 \begin{definition} \label{def:EW}
      Let $\cF$ be a finite set of branches on $S$, different from $L$. 
     The {\bf Eggers-Wall tree} $\Theta_L(\cF)$ of $\cF$ relative to $L$ 
     is the rooted tree obtained 
     as the quotient of the disjoint union of the individual Eggers-Wall trees $\Theta_L(A)$, 
     for $A \in \cF$,  by the equivalence relation generated by the following natural 
     gluing of $\Theta_L(A)$ with $\Theta_L(B)$ along the initial segments 
     $\ex_{L,A}^{-1}[0, k_L(A,B)]$ and $\ex_{L,B}^{-1}[0, k_L(A,B)]$:
         $$\ex_{L,A}^{-1}(\alpha) \sim \ex_{L,B}^{-1}(\alpha), \: \mbox{ for all } \:  
         \alpha \in [0, k_L(A,B)].$$
     One endows $\Theta_L(\cF)$ with the {\bf exponent function} 
     $\ex_L : \Theta_L(\cF) \to  [0, \infty]$ and the 
     {\bf index function} $\de_L: \Theta_L(\cF) \to \N$ obtained by 
     gluing the initial 
     functions $\ex_{L, A}$ and $\de_{L, A}$ respectively, for 
     $A$ varying among the elements of $\cF$. 
  \end{definition}
  
  It is an instructive exercise to prove that indeed the index functions of the branches 
  of $\cF$ get glued into a single index function on $\Theta_L(\cF)$. Note that, by 
  construction, $k_L(A, B) = e_L( A \wedge_L B)$ for any pair $(A,B)$ of distinct branches 
  of $\cF$. Here $A \wedge_L B$ denotes the infimum 
  of the leaves of $\Theta_L(\cF)$ labeled by $A$ and $B$, 
  relative to the partial order on the vertices of $\Theta_L(\cF)$ 
  defined by the root $L$ (see Notation \ref{not:notinf}).

 \begin{example} \label{EWmany}
      Consider a set $\cF$ of branches in $(\C^2, 0)$, whose elements 
     $C_i$ (where $i \in \{1, \dots, 5 \}$)  
     have the following Newton-Puiseux roots $ \eta_i$:
        $$\begin{array}{l}
             \eta_1 = x^2 \\
              \eta_2 =  x^{5/2} + x^{8/3} \\
             \eta_3 = x^{5/2} + x^{11/4}  \\
             \eta_4 = x^{7/2}  +  x^{17/4} \\
             \eta_5 = x^{7/2} + 2 x^{17/4} + x^{14/3}.
          \end{array} $$
     As before, we assume that $L$ is defined by $x=0$. 
    Then $k_L(C_1, C_2) = k_L(C_1, C_3) = k_L(C_1, C_4) = k_L(C_1, C_5) =2$, 
      $k_L(C_2, C_4) = k_L(C_2, C_5) = 5/2$, $k_L(C_2, C_3) = 8/3$, \linebreak 
      $k_L(C_3, C_4)=k_L(C_3, C_5)=5/2$, 
      $k_L(C_4, C_5) = 17/4$ 
    and the Eggers-Wall tree of $\cF$ relative to $L$ is as 
    drawn in Figure \ref{fig:EWfive}. 
 \end{example}

\begin{figure}[h!] 
\vspace*{6mm}
\labellist \small\hair 2pt 
\pinlabel{$L$} at 160 -10
\pinlabel{$C_1$} at 266 185
\pinlabel{$C_2$} at 300 315
\pinlabel{$C_3$} at 200 370
\pinlabel{$C_4$} at 110 370
\pinlabel{$C_5$} at 0 375

\pinlabel{$\mathbf{0}$} at 140 6
\pinlabel{$\mathbf{2}$} at 107 82
\pinlabel{$\mathbf{5/2}$} at 80 120
\pinlabel{$\mathbf{7/2}$} at 44 215
\pinlabel{$\mathbf{17/4}$} at 17 260
\pinlabel{$\mathbf{14/3}$} at -10 310
\pinlabel{$\mathbf{8/3}$} at 164 195
\pinlabel{$\mathbf{11/4}$} at 140 295

\pinlabel{$1$} at 150 50
\pinlabel{$1$} at 124 110
\pinlabel{$1$} at 200 125
\pinlabel{$1$} at 95 175

\pinlabel{$2$} at 132 165
\pinlabel{$2$} at 60 255
\pinlabel{$2$} at 165 245
\pinlabel{$4$} at 42 295
\pinlabel{$4$} at 80 300
\pinlabel{$4$} at 195 315
\pinlabel{$6$} at 220 245
\pinlabel{$12$} at 30 340

\endlabellist 
\centering 
\includegraphics[scale=0.70]{EWfive} 
\vspace*{5mm} 
\caption{The Eggers-Wall tree of Example \ref{EWmany}} 
\label{fig:EWfive}
\end{figure}

\begin{remark}  \label{rem:ultradual}
       Eggers \cite{E 83} had constructed a tree 
      which is nearly homeomorphic to the tree of Definition \ref{def:asstree} (there may 
      be a difference in the neighborhood of their roots, which the interested reader 
      may find without difficulties). He did not consider on it the index function. 
      Instead, he considered two types of edges which, given the exponent 
      function,  allowed to encode the same information as the index function.  
      Another difference with Definition \ref{def:EW} is that Eggers assumed that 
      the smooth branch $L$ is transversal to the tangents of all the branches of 
      $\cF$. What we call \emph{Eggers-Wall tree} was introduced by Wall \cite{W 03} 
      in his reinterpretation of Eggers' work using computations of $0$-chains and 
      $1$-chains supported on the tree. 
 \end{remark}

 \begin{remark} 
      Before Eggers' work, Kuo and Lu \cite{KL 77} had introduced a related tree 
      associated to a finite set $\cF$ of branches on $(\C^2, 0)$, different from the $y$-axis. 
      Namely, they represented by a segment each \emph{Newton-Puiseux root} of 
      the branches of $\cF$, exactly as in Definition \ref{def:EWirr}. 
      Then the construction proceeded exactly like for 
      Eggers-Wall trees, with a slight variant. Namely, they used a general graphical convention 
      for building \emph{dendrograms} in genetics, using only horizontal and vertical segments. 
      If one proceeds instead as in our Definition \ref{def:EW}, one may prove that 
      the Galois group of the extension of $\C[[x]]$ obtained by adjoining the roots 
      of $\cF$ acts on the resulting tree, and that its quotient by this action is the 
      Eggers-Wall tree $\Theta_L(\cF)$. Moreover, the index of each point of 
      $\Theta_L(\cF)$ is the cardinal of the fiber of this quotient map. 
 \end{remark}

The fact that in the previous notations $\Theta_L(\cF), \ex_L, \de_L$ we mentioned 
only the dependency on $L$, and not the whole coordinate system $(x,y)$, 
comes from the following fact: 

\begin{proposition}  \label{smoothcv}
 The Eggers-Wall tree  $\Theta_L(\cF)$, 
 seen as a rooted tree endowed with the exponent function $\ex_L$ and the index  
 function $\de_L$, depends only on the pair $(\cF, L)$, where $L$ is defined 
 by $x=0$.
\end{proposition}

\begin{proof}
     $\bullet$ \emph{Assume first that $A$ is a branch distinct from $L$.}
     
    Choose some $p \in \N^*$. Let $\phi_p : \tilde{S} \to S$ be the cyclic cover of $S$ of 
    degree $p$, ramified along $L$. Denote by $\tilde{O}\in \tilde{S}$ the preimage of 
    $O\in S$. 
    Consider then the (total) pullback $\phi_p^* A$. By computing 
    in the coordinates $(x,y)$, with respect to which the morphism $\phi_p$ is simply given by 
    $x= u^p, y = v$ for suitably chosen coordinates $(u,v)$ on $\tilde{S}$, 
    one sees that this pullback has only smooth branches if and only if 
    $p$ is divisible by $A \cdot L$. 
    
    Suppose that this is the case. Then again by computing in local coordinates, one sees 
    that the set $\cE(A)$ of characteristic exponents of $A$ with respect to $(x,y)$ 
    is equal to the set of rationals of the form $\frac{1}{p} A_i \cdot_{\tilde{O}} A_j$, where 
    $(A_i, A_j)$ varies among the set of couples of distinct branches of $\phi_p^* A$ 
    and the intersection numbers are computed at $\tilde{O}$.    
    This shows that $\cE(A)$ depends only on the pair $(A,L)$, and not on the 
    coordinate system $(x,y)$ chosen such that $L$ is defined by  $x=0$. 
    
    \medskip
    $\bullet$ \emph{Assume now that $A$ and $B$ are two different branches distinct from $L$.}
    Take $p \in \N^*$ divisible by both $A \cdot L$ and $B \cdot L$. Again by computing 
    in the coordinates $(x,y)$, we see that $k_L(A,B)$ is the maximal value of the 
    rational numbers of the form $\frac{1}{p} A_i \cdot_{\tilde{O}} B_j$, 
    where $A_i$ varies among the 
    branches of $\phi_p^* A$ and $B_j$ among those of $\phi_p^*  B$. 
    This shows that the order of coincidence 
    of $A$ and $B$ with respect to $(x,y)$ depends also only on $L$. 
    
    \medskip
    By combining the two invariance properties we deduce the proposition. 
\end{proof}

Let us introduce a third function defined on the Eggers-Wall tree:

\begin{definition} \label{Intcoef}  
    Let $A$ be a branch  on $S$. 
    The {\bf contact complexity}  
    ${\ic}_L(P)$ of a point $P \in \Theta_L(A)$ 
    is defined by:
       $${\ic}_L(P) :=   \left(\sum_{j =1}^{l} \frac{\alpha_j - \alpha_{j-1}}{{\de}_{j-1}}\right)  + 
          \frac{{\ex}_L(P)- \alpha_l}{{\de}_l}.$$
     where $\alpha_0:=0$, $\alpha_1<  ... < \alpha_g$ are the characteristic exponents 
     of $A$ relative to $L$, that ${\de}_j :=  \de(\alpha_0, \dotsc, \alpha_j)$ is 
     the value of the index function ${\de}_L$ in restriction to the half-open interval  
     $ ( P_j = \ex_L^{-1}(\alpha_j) , P_{j+1} = \ex_L^{-1}(\alpha_{j+1})]$ 
     and ${\ex}_L(P) \in [\alpha_l, \alpha_{l +1}]$ 
   is the value of the exponent function 
     at the point $P$. The possibility  $\alpha_l=\alpha_0=0$ is  allowed.   
\end{definition}

\begin{remark}
    The previous definition gives the same value to ${\ic}_L(P)$ when 
    ${\ex}_L(P) = \alpha_l$, if we compute it by looking at $\alpha_l$ 
    either as an element of $[ \alpha_{l-1}, \alpha_l]$ or as an element 
    of $[ \alpha_{l}, \alpha_{l+1}]$. 
\end{remark}

Note that the right-hand side of the formula defining ${\ic}_L(P)$ may be reinterpreted 
as an integral of the piecewise constant function $1/ {\de}_L$  along the segment 
$[LP]$ of $\Theta_L(A)$, the measure being the one determined by the exponent 
function:

\begin{equation}  \label{intcoefint}
    {\ic}_L(P) = \int_L^{P} \frac{d \: {\ex}_L}{{\de}_L}. 
\end{equation}

This allows to express conversely ${\ex}_L$ in terms of ${\ic}_L(P)$ and ${\de}_L$:

\begin{equation}  \label{expfunctint}
    {\ex}_L(P) = \int_L^{P} {\de}_L \:  d \: {\ic}_L  . 
\end{equation}

\begin{remark}   \label{rem:integform}
   Formulae (\ref{intcoefint})  and (\ref{expfunctint}) are inspired by the 
   formulae (3.7) and (3.9) of Favre and Jonsson's book \cite{FJ 04}, 
   relating \emph{thinness} and \emph{skewness} 
   as functions on the valuative tree. 
\end{remark}

As the function $\de_L : \Theta_L(A) \to \N^*$ is increasing along the segment 
$\Theta_L(A)$, formulae (\ref{intcoefint})  and (\ref{expfunctint}) imply:

\begin{corollary} \label{convconc}
    The contact complexity $\ic_L$ is an increasing homeomorphism from 
    $\Theta_L(A)$ to $[0, \infty]$. Moreover, it is piecewise affine and 
    concave in terms of the parameter $\ex_L$. Conversely, the function $\ex_L$ 
    is continuous piecewise affine and 
    convex in terms of the parameter $\ic_L$. 
\end{corollary}

Let us consider now the case of a finite set $\cF$ of branches. 
As an easy consequence of Definition \ref{Intcoef}, we get:

\begin{lemma} \label{contint} $\,$ 
  The contact functions of the branches of $\cF$ glue 
        into a continuous strictly increasing surjection $\ic_L : \Theta_L(\cF) \to [0, \infty]$.
\end{lemma}

This allows us to formulate the following definition:

\begin{definition} \label{Intcoefgen}
    Assume that both $S$ and $L$ are smooth. 
    Let $\cF$ be a finite set of branches on $S$. 
    The {\bf contact complexity}
     ${\ic}_L : \Theta_L(\cF) \to [0, \infty]$ 
    is the function obtained by gluing the contact complexities  
    of the individual branches of $\cF$. 
\end{definition}

We chose the name of this function motivated by the following theorem, which is a 
reformulation of a result of Max Noether's paper \cite{N 90} 
(see also \cite[Theorem 4.1.6]{W 04}):

\begin{theorem}  \label{intcomp}
     Let $A$ and $B$ be two distinct branches of $\cF$. 
     Then, in terms of the partial order defined by the root $L$ on the set of vertices 
     of $\Theta_L(\cF)$, one has:
       $$U_L (A, B) = {\ic}_L(A \wedge_L B)^{-1} .$$
\end{theorem}

As a consequence, we get the following strengthening of P\l oski's theorem 
(what is stronger is the fact that the end-rooted tree associated to the ultrametric 
$U_L$ is isomorphic to the Eggers-Wall tree $\Theta_L(\cF)$):

\begin{theorem}  \label{thm:EgWallisom}
     Let $L$ be a smooth branch on a smooth germ of surface $S$. Consider 
     a finite set $\cF$ of branches on $S$, distinct from $L$. Then 
     $U_L$ is an ultrametric on $\cF$ and the associated end-rooted tree is 
     isomorphic to the Eggers-Wall  tree $\Theta_L(\cF)$ of $\cF$ relative to $L$.
\end{theorem}

\begin{proof}
    By Lemma \ref{contint}, $\ic_L $ restricts to a height function on the rooted tree 
     $\Theta_L(\cF)$. Therefore, its inverse $\ic_L^{-1}$ is a depth 
    function. By Lemma \ref{lem:depthultra}, we deduce that $U_L$ is an 
    ultrametric. Using then Proposition \ref{prop:reconstr}, applied to the 
    end-rooted tree $\Theta_L(\cF)$ depth-dated by $\ic_L^{-1}$, we 
    deduce that the end-rooted tree associated to the ultrametric $U_L$ on 
    $\cF$ is indeed isomorphic to $\Theta_L(\cF)$. 
\end{proof}

As a consequence of the previous theorem and of Theorem \ref{thm:topint}, we get:

\begin{theorem}  \label{thm:EWcvx}
     Let $L$ be a smooth branch on a smooth germ of surface $S$. Consider 
     a finite set $\cF$ of branches on $S$, distinct from $L$. Then the 
     Eggers-Wall tree $\Theta_L(\cF)$ is isomorphic as a rooted tree with the 
     convex hull of $\{L\} \cup \cF$ in the dual graph of an embedded 
     resolution of their sum, rooted at the strict transform of $L$. 
\end{theorem}

\begin{remark} \label{initisom}
    It was the third author who proved first an isomorphism theorem of this kind 
    in \cite[Theorem 4.4.1]{P 01}. There the isomorphism was proved by embedding 
    the two trees in a common space and proving that the images coincided. In that work  
    a slightly different notion of Eggers-Wall tree was used, coinciding topologically 
    with Eggers' original definition from \cite{E 83}. Later, \cite[Theorem 4.4.1]{P 01} 
      was refined by Wall 
      \cite[Sect. 9.4]{W 04} (see also \cite[Sect. 9.10, Rem. on Sect. 9.4]{W 04}). 
     Let us mention also that with the hypothesis of Theorem \ref{thm:EWcvx}, 
     the Eggers-Wall tree $\Theta_{L} (\cF)$ is combinatorially 
     isomorphic to the dual graph 
     of a \emph{partial} embedded resolution of $L + \cF$ (see  \cite[Section 3.4]{GP 03}).
\end{remark}

\begin{remark} \label{Campillo} 
     When both $S$ and $L$ are smooth and $\cF$ is a finite set of branches 
     on $S$ different from $L$, the fact that $U_L$ is an 
  ultrametric distance on $\cF$ even when the base field has positive characteristic 
   was proved before by the first author and P\l oski in \cite[Theorem 2.8]{GP 15}. 
  The associated end-rooted tree provides, in a way, a generalization of the notion of 
  characteristic exponents in positive characteristic introduced in Campillo's book (see     
  \cite[Chapter III]{C 80}). As noted in the introduction, our approach also works 
  for arborescent surface singularities $S$ defined over algebraically closed fields in positive 
  characteristic, therefore $U_L$ is an ultrametric also in this generality, for any 
  branch $L$ on $S$. 
\end{remark}

\section{Valuative considerations}
\label{sec:valcons}

In this section we recall first the notions of \emph{valuation} and \emph{semivaluation} on 
the local ring of $S$  and the natural partial order on the set of all 
semivaluations. We introduce the \emph{order valuations} defined by irreducible 
exceptional divisors and the \emph{intersection semivaluations} defined by the branches 
lying on $S$. The choice of a fixed branch $L$ on $S$ allows 
to define versions of the previous (semi)valuations which are 
\emph{normalized relative to $L$}. We prove then that, for 
arborescent singularities, two such normalized (semi)valuations 
are in the same order relation as their representative points in the dual tree of an 
embedded resolution of them and of the branch $L$, seen as a tree rooted at $L$.

\subsection{Basic types of valuations and semivaluations}
$\:$ 
\medskip

In this subsection we define the types of valuations and semivaluations 
considered in the sequel. We do not assume here that the normal surface singularity $S$ is 
arborescent.
\medskip

Denote by $\cO$ the local ring of $S$ and by $\mathfrak{m}$ its maximal ideal. 
Denote also:
   $$ \overline{\R}_+ := \R_+ \cup \{+ \infty\} = [0, + \infty]$$
 endowed with the usual total order. 
 
In full generality, a \emph{valuation} or a \emph{semivaluation} takes its values 
in an arbitrary totally ordered abelian group enriched with a symbol $+ \infty$ which 
is greater than any element of the group. Here we will restrict to the special 
case where the totally ordered abelian group is $(\R, +)$:

\begin{definition}  \label{def:semival}
    A {\bf semivaluation} on $\cO$ is a map $\nu: \cO \to \overline{\R}_+$ such that:
       \begin{itemize}
           \item $\nu(gh) = \nu(g) + \nu(h)$ for all $g,h \in \cO$. 
           \item $\nu(g+ h)Ê\geq \min\{ \nu(g), \nu(h)\}$ for all $g,h \in \cO$. 
           \item $\nu(1) =0$ and $\nu(0) = + \infty$.
       \end{itemize}
    If  $\nu^{-1}(+ \infty) =\{0\}$, then one says that $\nu$ is a {\bf valuation}. 
    Denote by $\mathrm{Val}(S)$ the set of valuations of $\cO$ and 
    by $\overline{\mathrm{Val}}(S)$ the set of semivaluations.  There is a natural 
    partial order $\leq_{val}$ on $\overline{\mathrm{Val}}(S)$, defined by:
       $$\nu_1 \leq_{val} \nu_2 \:  \Longleftrightarrow \:   \nu_1(h) \leq \nu_2(h), \:  \forall \:  h \in \cO.$$
\end{definition}

In the sequel we will consider the following special types of valuations and semivaluations:

\begin{definition}  \label{def:exval}
   Let $L$ be a branch on $S$ and $\pi$ be an embedded resolution of it. 
   As usual, $(E_v)_{v \in \cV}$ denote the irreducible components of its exceptional divisor. 
   Let $A$ be a branch different from $L$. 
    \begin{enumerate}
        \item If $v \in \cV$, the {\bf $v$-order}, denoted by  
             $\mathrm{ord}_{v} : \cO \to \overline{\R}_+$,  
              is defined by:  
            $$\mathrm{ord}_{v}(h) := \mbox{ order of vanishing of }  \pi^*(h) \mbox{ along } E_v. $$
        \item
             If $v \in \cV$, the {\bf $v$-order relative to $L$}, denoted by  
             $\mathrm{ord}_{v}^L : \cO \to \overline{\R}_+$,    is defined by:  
               $$\mathrm{ord}_{v}^L(h) := \frac{\mathrm{ord}_{v}(h)}
                      {- \check{E}_v \cdot \check{E}_l}. $$
       \item The {\bf $A$-intersection order}, denoted by 
          $\mathrm{int}_A : \cO \to \overline{\R}_+$, is defined by:
          $$ \mathrm{int}_A(h) := \left\{ \begin{array}{lcl} 
                              A \cdot Z_h & \mbox{ if }  & A \mbox{ is not a branch of } Z_h \\
                              + \infty &  & \mbox{ otherwise.}
                                   \end{array} \right. $$
       \item The {\bf $A$-intersection order relative to $L$}, denoted by 
            $\mathrm{int}_A^L : \cO \to \overline{\R}_+$, is   defined by:
                   $$ \mathrm{int}_A^L(h) :=   \frac{\mathrm{int}_A(h)}{A \cdot L}. $$
    \end{enumerate}
\end{definition}

Note that the functions $\mathrm{ord}_v$ and $\mathrm{ord}_v^L$ are valuations, but that the 
functions $ \mathrm{int}_A$ and $\mathrm{int}_A^L$ are semivaluations which 
are not valuations. Indeed, they take the value $+ \infty$ on all the elements of 
$\mathcal{O}$ which vanish on the branch $A$. 

\begin{remark}  \label{rem:normaliz}
    In \cite{FJ 04}, a semivaluation $\nu$ of the local ring $\C[[x,y]]$ is called 
    \emph{normalized} relative to the variable $x$ if $\nu(x) =1$. If $L$ 
    denotes the branch $Z_x$ then, with our notations, $\mathrm{ord}^L$ and 
    $\mathrm{int}_A^L$ are normalized relative to $x$. In our more general 
    context of arbitrary normal surface singularities, the branch $L$ is not 
    necessarily a principal divisor. 
\end{remark}

\begin{remark}  \label{rem:skewness}
   In \cite[Sect. 7.4.8]{J 15}, in which $S$ is considered to be \emph{smooth}, 
   Jonsson defines a function $\alpha$ on the set of valuations of the local ring 
    $\mathcal{O}$ which are proportional to the divisorial valuations $\mbox{ord}_u$, 
    by the following formula:
       $ \alpha( t \cdot  \mbox{ord}_u) = t^2 (\check{E}_u \cdot \check{E}_u).$
     That is, $\alpha$ is homogeneous of degree $2$ and takes the value 
     $\check{E}_u \cdot \check{E}_u$ on the valuation $\mbox{ord}_u$. 
     In \cite[Sect. 7.6.2, Note 13]{J 15}, he remarks that 
     this function $\alpha$ is the \emph{opposite} of the \emph{skewness} function 
     denoted with the same symbol $\alpha$ in \cite{FJ 04}. A smooth germ 
     $S$ is arborescent and verifies $\det(S) =1$. Therefore, by Proposition 
     \ref{prop:equlfund}, his definition may be reexpressed in the following way 
     in the same case of smooth germs $S$:
        $$ \alpha( t \cdot \mbox{ord}_u) = - t^2 p(u,u) = - t^2 \det(S)^{-1} \cdot p(u,u).$$ 
     This indicates two possible generalizations of the function $\alpha$ to 
     arbitrary arborescent singularities, depending on which of the two previous 
     equalities is taken as a definition. 
   \end{remark}

\subsection{The valuative partial order for arborescent singularities}
$\:$ 
\medskip

The following theorem extends Lemma 3.69 of Favre and Jonsson \cite{FJ 04} from 
a smooth germ $S$ of surface and a smooth branch $L$ on it, to arborescent singularities and 
arbitrary branches on them: 

\begin{theorem}  \label{thm:ordsemival}
   Let $L, A$ be two distinct branches on the arborescent singularity $S$. 
   Let $\pi$ be an embedded resolution of $L + A$ and let $\Gamma_L(A)$ 
   be the dual tree of the total transform of $L + A$. Consider $\Gamma_L(A)$ as a   
   combinatorial tree 
   rooted at $L$ and let $\preceq_L$ be the corresponding partial order. 
   Assume that $u, v \in \cV$. Then:
     \begin{enumerate}
          \item  \label{ineqord} 
              $\mathrm{ord}_{u}^L \leq_{val} \mathrm{ord}_{v}^L$ if and only if 
                $u \preceq_L v$. 
                \smallskip
          \item   \label{ineqint} 
               $\mathrm{ord}_{u}^L \leq_{val} \mathrm{int}_A^L$ if and only if 
               $ u \preceq_L A$. 
     \end{enumerate}   
\end{theorem}

\begin{proof}
   The proof of this theorem is  strongly based on the determinantal 
   formula of Eisenbud and Neumann stated in Proposition \ref{prop:equlfund}. 
\medskip

\noindent 
 {\bf The proof of the implication $u \preceq_L v 
            \Longrightarrow \mathrm{ord}_{u}^L \leq_{val} \mathrm{ord}_{v}^L $ in point \ref{ineqord}. }
            Consider an arbitrary germ of function $h \in \frak{m}$.  We want to prove that 
            $\mathrm{ord}_{u}^L(h) \leq \mathrm{ord}_{v}^L(h) $.
            Let us work with an embedded resolution of $L + Z_h$. This is no reduction of 
            generality, as the truth of the relation $u \preceq_L v$ does not depend on the 
            resolution on which $E_u$ and $E_v$ appear as irreducible components of the 
            exceptional divisor. 
           As $\mathrm{ord}_{u}(h)$ is the coefficient of $E_u$ in the exceptional transform 
           $(\pi^* Z_h)_{ex}$ of $Z_h$, the expansion (\ref{eq:sense3}) shows that: 
              \begin{equation} \label{eq:ordexpr} 
                    \mathrm{ord}_{u}(h)=  \check{E}_u \cdot (\pi^*{Z_h})_{ex}. 
              \end{equation}
            The desired inequality $\mathrm{ord}_{u}^L(h) \leq \mathrm{ord}_{v}^L(h)$ becomes:
              $$ \frac{ \check{E}_u \cdot (\pi^*Z_h)_{ex}}{- \check{E}_u \cdot \check{E}_l} 
                   \leq  \frac{ \check{E}_v \cdot (\pi^* Z_h)_{ex}}{- \check{E}_v \cdot \check{E}_l}. $$
           The divisor $(\pi^* Z_h)_{ex}$ being a linear combination with non-negative coefficients 
           of the divisors $(-\check{E}_w)_{w \in \cV}$ (see Proposition \ref{prop:exprexcept}), it is 
           enough to prove that:
               $$ \frac{ - \check{E}_u \cdot \check{E}_w}{- \check{E}_u \cdot \check{E}_l} 
                   \leq  \frac{ - \check{E}_v \cdot \check{E}_w}{- \check{E}_v \cdot \check{E}_l} 
                     \mbox{ for all } w \in \cV.$$
            By Proposition \ref{prop:equlfund}, the previous inequality is equivalent to: 
              $$ p(l,v) \cdot p(u,w) \leq p(l,u) \cdot p(v,w).$$
            But this last inequality is true, as a consequence of Proposition \ref{prop:refine}. 
            Indeed, the inequality $u \preceq_L v$ implies that $u \in [lv] \cap [uw]$, therefore 
            $[lv] \cap [uw] \neq \emptyset$.

         \medskip   
     \noindent
        {\bf The proof of the implication  
         $\mathrm{ord}_{u}^L \leq_{val} \mathrm{ord}_{v}^L 
            \Longrightarrow u \preceq_L v$ in  point \ref{ineqord}. }
             Assume by contradiction that the inequality $u \preceq_L v$ is not true. 
             This means that $[lv] \cap [uu] = \emptyset$. 
             As $[lu] \cap [vu] \neq \emptyset$ (because $u$ belongs to this intersection), Proposition 
             \ref{prop:refine} implies the inequality:
                 $$p(l,v) \cdot p(u,u) > p(l,u) \cdot p(v,u),$$
             which may be rewritten, using Proposition \ref{prop:equlfund}, as:
                $$ \frac{ - \check{E}_u \cdot \check{E}_u}{- \check{E}_u \cdot \check{E}_l} 
                   >  \frac{ - \check{E}_v \cdot \check{E}_u}{- \check{E}_v \cdot \check{E}_l}.$$
             Therefore, whenever the positive rational numbers 
             $(\epsilon_w)_{w \in \cV \setminus \{u \} }$ are small enough, one has also the strict 
             inequality: 
               \begin{equation} \label{eq:choiceH}
                      \frac{ - \check{E}_u \cdot H}{- \check{E}_u \cdot \check{E}_l} 
                      >  \frac{ - \check{E}_v \cdot H}{- \check{E}_v \cdot \check{E}_l},
               \end{equation}
             where $H \in \Lambda_{\Q}$ is defined by:
                 $H := \check{E}_u +  \sum_{w \in \cV \setminus \{u \} } \epsilon_w \check{E}_w.$
              As $H \in \check{\sigma}$, Proposition \ref{prop:intrealis} shows that 
              there exists $n \in \N^*$ such that $-nH$ is the exceptional transform of 
              a principal divisor. Denote by $h \in \frak{m}$ a defining function of such a 
              divisor. Therefore, $-nH = (\pi^* Z_h)_{ex}$, and the inequality (\ref{eq:choiceH}) 
              implies:
              \begin{equation} \label{eq: uhvh}
                  \frac{ \check{E}_u \cdot (\pi^* Z_h)_{ex}}{- \check{E}_u \cdot \check{E}_l} 
                     >  \frac{ \check{E}_v \cdot (\pi^* Z_h)_{ex}}{- \check{E}_v \cdot \check{E}_l}. 
                    \end{equation}

               Using formula (\ref{eq:ordexpr}), this inequality may be rewritten as:
                 $ \mathrm{ord}_{u}^L(h) > \mathrm{ord}_{v}^L(h). $
               But this contradicts the hypothesis 
               $\mathrm{ord}_{u}^L \leq_{val} \mathrm{ord}_{v}^L$.                              
   
   \medskip
        \noindent
       {\bf The proof of the implication $u \preceq_L A 
            \Longrightarrow \mathrm{ord}_{u}^L \leq_{val} \mathrm{int}_{A}^L $ in  point \ref{ineqint}.}
            Consider an arbitrary germ of function $h \in \frak{m}$.  We want to prove that 
            $\mathrm{ord}_{u}^L(h) \leq \mathrm{int}_{A}^L(h) $.
            We may assume that we work with a resolution of $L + A + Z_h$. 
            By Proposition \ref{prop:intexcep},  we have that 
            $\mathrm{int}_{A}(h) = - (\pi^* A)_{ex}\cdot (\pi^* Z_h)_{ex}$. 
            Using Lemma \ref{lem:excdual}, we deduce the equality:
                \begin{equation} \label{eq:intform} 
                      \mathrm{int}_{A}(h) = \check{E}_a \cdot (\pi^* Z_h)_{ex},
                \end{equation}
             where $E_a$ denotes the unique component of the exceptional divisor $E$ which 
             intersects the strict transform of $A$. By Corollary \ref{cor:invint}, $A \cdot L = 
              - \check{E}_a \cdot \check{E}_l$. 
              Therefore, the desired inequality $\mathrm{ord}_{u}^L(h) \leq \mathrm{int}_{A}^L(h) $ 
              becomes:
                 $$ \frac{ \check{E}_u \cdot (\pi^* Z_h)_{ex}}{- \check{E}_u \cdot \check{E}_l} 
                   \leq  \frac{ \check{E}_a \cdot (\pi^* Z_h)_{ex}}{- \check{E}_a \cdot \check{E}_l}. $$
           As before, it is enough to prove that:
               $$ \frac{ - \check{E}_u \cdot \check{E}_w}{- \check{E}_u \cdot \check{E}_l} 
                   \leq  \frac{ - \check{E}_a \cdot \check{E}_w}{- \check{E}_a \cdot \check{E}_l}  
                     \mbox{ for all } w \in \cV.$$      
            By  Proposition \ref{prop:equlfund}, the previous inequality is equivalent to: 
              $$ p(l,a) \cdot p(u,w) \leq p(l,u) \cdot p(a,w).$$
            But this last inequality is true, as a consequence of Proposition \ref{prop:refine}. 
            Indeed, the inequality $u \leq_L A$ implies that $u \in [la] \cap [uw]$, therefore 
            $[lv] \cap [uw] \neq \emptyset$.

            \medskip
     \noindent
      {\bf The proof of the implication $\mathrm{ord}_{u}^L \leq_{val} \mathrm{int}_{A}^L 
            \Longrightarrow u \preceq_L A$ in  point \ref{ineqint}.}
            We reason again by contradiction, assuming that the inequality $u \preceq_L {A}$ 
            is not true. This means that $[la] \cap [uu] = \emptyset$. 
             As $[lu] \cap [au] \neq \emptyset$ (because $u$ belongs to this intersection), 
             Proposition \ref{prop:refine} implies that:
                 $$p(l,a) \cdot p(u,u) > p(l,u) \cdot p(a,u).$$
            Replacing $v$ by $a$ in the reasoning done in the proof of formula 
            (\ref{eq: uhvh}) above,  we arrive 
            at the following inequality:
               $$  \frac{ \check{E}_u \cdot (\pi^* Z_h)_{ex}}{- \check{E}_u \cdot \check{E}_l} 
                     >  \frac{ \check{E}_a \cdot (\pi^* Z_h)_{ex}}{- \check{E}_a \cdot \check{E}_l}. $$
            Combining it with formulae (\ref{eq:ordexpr}) and (\ref{eq:intform}), as well as 
            Proposition \ref{prop:intexcep}, it becomes:
              $\mathrm{ord}_{u}^L(h) >   \mathrm{int}_{A}^L(h).$
            But this contradicts the hypothesis $\mathrm{ord}_{u}^L \leq_{val} \mathrm{int}_{A}^L$. 
\end{proof}

By combining Theorem \ref{thm:ordsemival} with 
Theorem \ref{thm:topint}, we get: 

\begin{corollary}  \label{cor:valtree}
   Let $S$ be an arborescent  singularity and $\cF$ a finite set of branches  
   on it. Let $L$ be a branch not belonging to $\cF$. 
   Consider any embedded resolution $\pi$ of the sum $D$ of $L$ with the 
   elements of $\cF$. Let $(E_u)_{u \in \cV}$ be the components  of the exceptional 
   divisor of $\pi$. Then the partial 
   order $\preceq_{val}$ is arborescent in restriction to the set:
      $$ \{\mathrm{ord}_u^L \:  | \:   u \in \cV \} \cup \{ \mathrm{int}_{A}^L \: | \:  A \in \cF \}$$
   and the associated extended rooted tree (in the sense of Definition \ref{def:rt}) 
   is isomorphic with the convex hull of $\{ L \} \cup \cF$ in the dual tree of the 
   total transform of $D$ by $\pi$, rooted at the strict transform of $L$. 
\end{corollary}

\begin{remark} \label{rem:birat}
Till now we worked with fixed embedded resolutions of the various reduced divisors 
considered on $S$. But, given a fixed divisor, one could consider the 
projective system of all its resolutions. One gets an associated direct system of  
embeddings of the dual graphs of total transforms. The associated ultrametric spaces 
are instead the same. Consider the set of all 
reduced divisors on $S$, directed by inclusion. One gets an associated direct 
system of isometric 
embeddings of ultrametric spaces, therefore of  isometric embeddings  
of associated trees. 
One could prove then an analog of Jonsson's \cite[Theorem 7.9]{J 15} 
(which concerns only \emph{smooth} germs $S$), which 
presents a valuative tree associated to the singularity $S$ (that is, a quotient 
of a Berkovich space) as a projective 
limit of dual trees. We prefer not to do this here, in order to restrict 
to phenomena visible on fixed resolutions of $S$ and which may be described 
by elementary combinatorial means, without any appeal to Berkovich geometry. 
\end{remark}

\begin{remark}   \label{RuGi}
  After having seen a previous version of this paper, Ruggiero sent us the preliminary 
  version \cite{GR 17} of the paper he writes with Gignac. In that paper they 
  extend to the spaces $\overline{\mathrm{Val}}(S)$ of semivaluations 
  of normal surface singularities $S$, part of the theory described 
  in \cite{FJ 04} and \cite{J 15}. This started our collaboration with Ruggiero 
  leading to the sequel \cite{GGPR 17} of the present paper.
\end{remark}

\section{Perspectives on non-arborescent singularities}
\label{sec:oprob}

In this section we give two examples, showing that for singularities 
which are not necessarily arborescent, $U_L$ is not necessarily an ultrametric 
or even a metric on the set of branches distinct from $L$. Then we state some open 
problems related with this phenomenon.

\medskip
\subsection{Non-arborescent examples}
$\:$

\begin{example} \label{ex:non-arborescent}
Consider the weighted dual graph $\Gamma$ represented in Figure \ref{X1} 
(the self-intersections being indicated between brackets and the genera  
being arbitrary). Denote by $I$ the associated intersection form. 
Consider the matrix of $-I$, obtained after having ordered 
the vertices $a,b,c,d$. By computing its principal minors, one sees that 
this symmetric matrix is positive definite, which shows that $I$ is 
negative definite. By a theorem of Grauert \cite[Page 367]{G 62} 
(see also Laufer \cite[Theorem 4.9]{L 71}), any 
divisor with normal crossings in a smooth complex surface which admits 
this weighted dual graph may be contracted to a normal surface singularity $S$.  
The graph $\Gamma$ 
admitting cycles, the singularity is not arborescent. Denote by 
$\pi$ the resolution of $S$ whose dual graph is $\Gamma$. 

Let us consider branches $A, B, C, L$ on $S$ whose strict transforms by $\pi$ 
are smooth and intersect transversally $E_a, E_b, E_c, E_l$ at smooth points 
of the total exceptional divisor $E$.  Therefore $(\pi^* A)_{ex}= -\check{E}_a$, 
$(\pi^* B)_{ex} = -\check{E}_b$, 
$(\pi^* C)_{ex} = -\check{E}_c$, and $(\pi^* L)_{ex} =  - \check{E}_l$.  

The entries $-(\check{E}_u\cdot \check{E}_v) \cdot \det(S)$ 
of the adjoint matrix of $(- E_u \cdot E_v)_{uv}$ are as  indicated in Figure \ref{X2}. 
Using Corollary \ref{cor:invint}, one may compute then the values of $U_L$, getting: 
  \begin{eqnarray*}
            \det(S) \cdot U_L(A,B)  
            = \dfrac{64 \cdot 70 \cdot 114}{6272}, \,   \\
            \det(S)  \cdot  U_L(A,C) 
             = \dfrac{64 \cdot 70 \cdot 114}{6440} , \,   \\
            \det(S)  \cdot  U_L(B,C) 
            = \dfrac{64 \cdot 70 \cdot 114}{6384} .
       \end{eqnarray*}
  
 As the three values are pairwise distinct, we see that $U_L$ is not an ultrametric 
 on the set $\{A, B, C\}$. Therefore it is nor an ultrametric on  the set 
  of branches $\mathcal{B}(S) \setminus \{ L \}$. 
  Let us mention that $\det(S) = 56$, even if one does not need this in order to do the 
  previous computations. 
  One may check immediately on the above values 
  that $U_L$ is nevertheless a metric on the set $\{A, B, C\}$. We do not 
  know if it is also a metric on $\mathcal{B}(S) \setminus \{ L \}$.
\end{example}

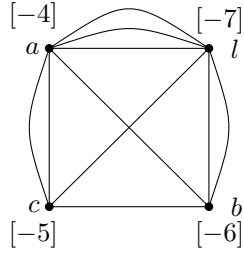
\begin{figure}
\centering
\begin{tikzpicture} [scale=0.7]

 \draw (0,0) .. controls (1.5 ,1) .. (3,0);
 \draw (0,0) .. controls (1.5,0.5) .. (3,0);
 \draw (0,0) .. controls (1,0) .. (3,0);

 \draw (0,0) .. controls (-0.5,-1.5) .. (0,-3);
 \draw (0,0) .. controls (0,-1) .. (0,-3);

  \draw (3,0) .. controls (3,-1) .. (3,-3);
 \draw (3,0) .. controls (3.5,-1.5) .. (3,-3);
 
 \draw (0,0) -- (3,-3);
  \draw (3,0) -- (0,-3);
  \draw (0,-3) -- (3,-3);
  
\node[draw,circle,inner sep=1pt,fill=black] at (0,0) {};

\node[draw,circle,inner sep=1pt,fill=black] at (3,0) {};

\node[draw,circle,inner sep=1pt,fill=black] at (0,-3) {};

\node[draw,circle,inner sep=1pt,fill=black] at (3,-3) {};

\draw (0 ,0) node[left]{ \hspace{-0,12cm} $a$ } ;

\draw (0 ,-3) node[left]{ \hspace{-0,12cm} $c$} ;

\draw (3 ,-3) node[right]{ \hspace{0,05cm} $b$ } ;

\draw (3 ,0) node[right]{ \hspace{0,05 cm} $l$} ;

\draw (-0.3 ,0.2) node[above]{  $[-4]$} ;
\draw (3.2 ,0.1) node[above]{  $[-7]$} ;
\draw (-0.3 ,-3.1) node[below]{ $[-5]$} ;
\draw (3.2 ,-3.1) node[below]{ $[-6]$} ;

\end{tikzpicture}
\caption{The dual graph of the singularity in  Example \ref{ex:non-arborescent}.} 
\label{X1}
\end{figure}

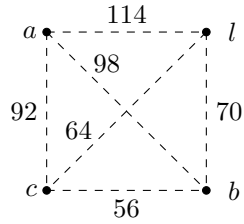
\begin{figure}
\centering
\begin{tikzpicture} [scale=0.7]

   \draw [dashed] (0,0) .. controls (1,0) .. (3,0);
 \draw [dashed]  (0,0) .. controls (0,-1) .. (0,-3);

  \draw [dashed]  (3,0) .. controls (3,-1) .. (3,-3);
 
 \draw [dashed]  (0,0) -- (3,-3);
  \draw [dashed]  (3,0) -- (0,-3);
  \draw [dashed]  (0,-3) -- (3,-3);
  
\node[draw,circle,inner sep=1pt,fill=black] at (0,0) {};

\node[draw,circle,inner sep=1pt,fill=black] at (3,0) {};

\node[draw,circle,inner sep=1pt,fill=black] at (0,-3) {};

\node[draw,circle,inner sep=1pt,fill=black] at (3,-3) {};

\draw (0 ,0) node[left]{ \hspace{-0,12cm} $a$ } ;

\draw (0 ,-3) node[left]{ \hspace{-0,12cm} $c$} ;

\draw (3 ,-3) node[right]{ \hspace{0,05cm} $b$ } ;

\draw (3 ,0) node[right]{ \hspace{0,05 cm} $l$} ;

 \draw (1.5 ,0) node[above]{  $114$} ; 
 \draw (0 , -1.5) node[left]{$92$} ; 
\draw (1.5 ,-3) node[below]{ $56$} ;
\draw (3 ,-1.5) node[right]{ $ 70$} ; 
\draw (0.7,-0.6) node[right]{ $98$} ; 
\draw (0.6,-2.2) node[above]{ $64$} ; 
\end{tikzpicture}
\caption{The label on the edge $[u v]$ is the number 
      $-(\check{E}_u\cdot \check{E}_v) \cdot \det(S)$ in Example \ref{ex:non-arborescent}. } 
   \label{X2}
\end{figure}

\begin{example} \label{ex:non-metric}
Consider the weighted graph $\Gamma$ represented in Figure \ref{X3} 
(the self-intersections being indicated between brackets and the genera  
being arbitrary). As in the previous example, we see that there exist 
normal surface singularities with such weighted dual graphs, and that 
they are not arborescent. 

Let us consider branches $A, B, C, L$ on $S$ with the same properties 
as in the previous example. 

The entries $-(\check{E}_u\cdot \check{E}_v) \cdot \det(S)$ 
of the adjoint matrix of $(- E_u \cdot E_v)_{uv}$ are as  indicated in Figure \ref{X4}. 
Using Corollary \ref{cor:invint}, one may compute then the values of $U_L$, getting: 
  \[
            \det(S) \cdot U_L(A,B) = 
            75, \quad
            \det(S) \cdot U_L(A,C) =  
             35 , \quad
            \det(S) \cdot U_L(B,C) = 
            35 .
     \]
 One sees that in this case $U_L$ is not even a metric on the set $\{A,B,C\}$. 
\end{example}

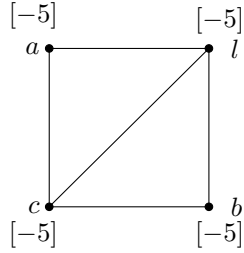
\begin{figure}
\centering
\begin{tikzpicture}  [scale=0.7]

 \draw (0,0) .. controls (1,0) .. (3,0);

 \draw (0,0) .. controls (0,-1) .. (0,-3);

  \draw (3,0) .. controls (3,-1) .. (3,-3);

  \draw (3,0) -- (0,-3);
  \draw (0,-3) -- (3,-3);
  
\node[draw,circle,inner sep=1pt,fill=black] at (0,0) {};

\node[draw,circle,inner sep=1pt,fill=black] at (3,0) {};

\node[draw,circle,inner sep=1pt,fill=black] at (0,-3) {};

\node[draw,circle,inner sep=1pt,fill=black] at (3,-3) {};

\draw (0 ,0) node[left]{ \hspace{-0,12cm} $a$ } ;

\draw (0 ,-3) node[left]{ \hspace{-0,12cm} $c$} ;

\draw (3 ,-3) node[right]{ \hspace{0,05cm} $b$ } ;

\draw (3 ,0) node[right]{ \hspace{0,05 cm} $l$} ;

\draw (-0.3 ,0.2) node[above]{  $[-5]$} ;
\draw (3.2 ,0.1) node[above]{  $[-5]$} ;
\draw (-0.3 ,-3.1) node[below]{ $[-5]$} ;
\draw (3.2 ,-3.1) node[below]{ $[-5]$} ;

\end{tikzpicture}
\caption{The dual graph of the singularity in Example \ref{ex:non-metric}.} \label{X3}
\end{figure}

\begin{figure}
\centering
\begin{tikzpicture} [scale=0.7]

   \draw [dashed]  (0,0) .. controls (1,0) .. (3,0);
   \draw [dashed]  (0,0) .. controls (0,-1) .. (0,-3);

  \draw [dashed] (3,0) .. controls (3,-1) .. (3,-3);
 
 \draw [dashed]  (0,0) -- (3,-3);
  \draw [dashed]  (3,0) -- (0,-3);
  \draw [dashed]  (0,-3) -- (3,-3);
  
\node[draw,circle,inner sep=1pt,fill=black] at (0,0) {};

\node[draw,circle,inner sep=1pt,fill=black] at (3,0) {};

\node[draw,circle,inner sep=1pt,fill=black] at (0,-3) {};

\node[draw,circle,inner sep=1pt,fill=black] at (3,-3) {};

\draw (0 ,0) node[left]{ \hspace{-0,12cm} $a$ } ;

\draw (0 ,-3) node[left]{ \hspace{-0,12cm} $c$} ;

\draw (3 ,-3) node[right]{ \hspace{0,05cm} $b$ } ;

\draw (3 ,0) node[right]{ \hspace{0,05 cm} $l$} ;

 \draw (1.5 ,0) node[above]{  $30$} ; 
 \draw (0 , -1.5) node[left]{$30$} ; 
\draw (1.5 ,-3) node[below]{ $30$} ;
\draw (3 ,-1.5) node[right]{ $ 30$} ; 
\draw (0.7,-0.6) node[right]{ $12$} ; 
\draw (0.6,-2.2) node[above]{ $35$} ; 
\end{tikzpicture}
\caption{The label on the edge $[u v]$ is the number 
$-(\check{E}_u\cdot \check{E}_v) \cdot \det(S)$ in Example \ref{ex:non-metric}. } \label{X4}
\end{figure}
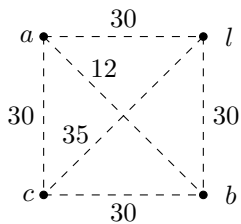

\medskip

\subsection{Some open problems}
$\:$ 
\medskip

Let us end this paper with some open problems: 
\begin{enumerate}
\item[1.]  \emph{Characterize the normal surface singularities for which 
    $U_L$ is a metric (compare with Examples \ref{ex:non-arborescent} and 
    \ref{ex:non-metric}). }
\item[2.] \emph{Characterize the normal surface singularities whose generic 
   hyperplane section is irreducible (compare with Theorem \ref{thm:ultraint}). }
\item[3.] \emph{Characterize the normal surface singularities  
    for which   $U_O$ is an ultrametric (compare with Theorem \ref{thm:ultraint}). }
\item[4.] \emph{Characterize the normal surface singularities for which 
    $U_O$ is a metric.}
\end{enumerate}

\bigskip
{\bf Acknowledgements.}
      This research was partially supported 
by the  French grant ANR-12-JS01-0002-01 SUSI and 
Labex CEMPI (ANR-11-LABX-0007-01), and also by  the Spanish 
Projects MTM2016-80659-P, 
MTM2016-76868-C2-1-P.
    The third author is grateful to Mar\'{\i}a Angelica Cueto, 
    Andr\'as N\'emethi and Dmitry Stepanov for inspiring conversations. 
    We are also grateful to Nicholas Duchon for having sent us his thesis, and to 
    Charles Favre, Mattias Jonsson, Andr\'as N\'emethi, 
    Walter Neumann and Matteo Ruggiero for their comments 
    on a previous version of this paper.

\end{document}